\documentclass[12pt,reqno]{amsart}

\usepackage[T1]{fontenc}
\usepackage{enumitem}
\usepackage{hyperref}
\hypersetup{
  colorlinks   = true,
  citecolor    = blue,
  linkcolor    = blue
}

\usepackage{amsmath,amsfonts,amsthm,color,tikz, comment, xcolor, xfrac,amssymb}
\usepackage[normalem]{ulem}
\usepackage{mathrsfs}

\usepackage{pdfsync}
\usepackage[font={scriptsize}]{caption}
\usepackage{environ}

\usepackage[left=1in, right=1in, top=1.1in,bottom=1.1in]{geometry}
\newcommand{\der}{\delta}

\newcommand{\hb}[1]{\textcolor{blue}{#1}}
\newcommand{\hp}[1]{{\textcolor{purple}{#1}}}

\newcommand{\iot}{\int_{0}^{t}}

\newcommand{\pt}{\partial}

\newcommand{\bzeta}{{\zeta}}

\DeclareMathOperator*{\argsup}{arg\,sup}

\newcommand{\R}{\mathbb R}
\newcommand{\N}{\mathbb N}

\newcommand{\bp}{\mathbf{P}}

\newcommand{\ca}{\mathcal A}

\newcommand{\cac}{\mathcal C}

\newcommand{\cf}{\mathcal F}
\newcommand{\cg}{\mathcal G}

\newcommand{\ci}{\mathcal I}
\newcommand{\cj}{\mathcal J}
\newcommand{\ck}{\mathcal K}

\newcommand{\cn}{\mathcal N}
\newcommand{\cp}{\mathcal P}
\newcommand{\cq}{\mathcal Q}
\newcommand{\cs}{\mathcal S}
\newcommand{\ct}{\mathcal T}
\newcommand{\cv}{\mathcal V}

\newcommand{\scrc}{\mathcal{T}}
    
\newcommand{\al}{\alpha}

\newcommand{\ep}{\varepsilon}

\newcommand{\ga}{\gamma}
\newcommand{\gga}{\Gamma}
\newcommand{\ka}{\kappa}
\newcommand{\la}{\lambda}
\newcommand{\laa}{\Lambda}

\newcommand{\si}{\sigma}

\newcommand{\vp}{\varphi}

\newcommand{\lp}{\left(}
\newcommand{\rp}{\right)}
\newcommand{\lc}{\left[}
\newcommand{\rc}{\right]}
\newcommand{\lcl}{\left\{}
\newcommand{\rcl}{\right\}}
\newcommand{\lln}{\left|}
\newcommand{\rrn}{\right|}

\newtheorem{theorem}{Theorem}[section]

\newtheorem{corollary}[theorem]{Corollary}

\newtheorem{definition}[theorem]{Definition}

\newtheorem{hypothesis}[theorem]{Hypothesis}
\newtheorem{lemma}[theorem]{Lemma}
\newtheorem{notation}[theorem]{Notation}

\newtheorem{proposition}[theorem]{Proposition}

\theoremstyle{remark}
\newtheorem{remark}[theorem]{Remark}

\theoremstyle{remark}

\newcommand{\bean}{\begin{eqnarray*}}
\newcommand{\eean}{\end{eqnarray*}}
\newcommand{\ben}{\begin{enumerate}}
\newcommand{\een}{\end{enumerate}}
\newcommand{\beq}{\begin{equation}}
\newcommand{\eeq}{\end{equation}}

\begin{document}

\title[Pathwise Relaxed Optimal Control]{Pathwise Relaxed Optimal Control of Rough Differential Equations}

\author[P. Chakraborty]{Prakash Chakraborty}
\address{P. Chakraborty: Marcus Department of Industrial and Manufacturing Engineering,
The Pennsylvania State University, State College}
\email{prakashc@psu.edu}

\author[H. Honnappa]{Harsha Honnappa}
\address{H. Honnappa: School of Industrial Engineering, Purdue University, West Lafayette}
\email{honnappa@purdue.edu}

\author[S. Tindel]{Samy Tindel}
\address{S. Tindel: Department of Mathematics, 
Purdue University, West Lafayette}
\email{stindel@purdue.edu}

\maketitle

\begin{abstract}
This note lays part of the theoretical ground for a definition of differential systems modeling reinforcement learning in continuous time non-Markovian rough environments. Specifically we focus on optimal relaxed control of rough equations (the term \emph{relaxed} referring to the fact that controls have to be considered as measure valued objects). With reinforcement learning in view, our reward functions encompass forms that involve an entropy-type term favoring exploration. In this context, our contribution focuses on a careful definition of the corresponding relaxed Hamilton-Jacobi-Bellman (HJB)-type equation. A substantial part of our endeavor consists in a precise definition of the notion of test function and viscosity solution for the rough relaxed PDE obtained in this framework. Note that this task is often merely sketched in the rough viscosity literature, in spite of the fact that it gives a proper meaning to the differential system at stake. In the last part of the paper we prove that the natural value function in our context solves a relaxed rough HJB equation in the viscosity sense.
\end{abstract}

\section{Introduction}
This note forms the first stepping stone towards a proper definition of continuous time differential systems modeling reinforcement learning in a rough non-Markovian environment.
In this paper we are concerned with one key ingredient towards this aim, namely pathwise optimal \emph{relaxed} control of rough differential equations. In order to describe the kind of system we are aiming at, consider a path $(x^\gamma_t : t \geq 0)$ which represents the state of a controlled process.  Let also $(\gamma_t : t \geq 0)$ be a curve of relaxed controls; 
that is, a path taking values in the set $\cp(U)$ of probability measures on a given control space $U$.
We are given a general noisy environment, which is modeled by a rough path $(\zeta_{t} : t \geq 0)$.
The state process $x^{\ga}$ is then the solution of the following rough differential equation,
\begin{equation}\label{d1}
x_t^{\ga} = x_0 + \iot \int_{U} b(s, x_s^{\ga}, a) \, \ga_s(da) ds + \iot \si(s, x_s^{\ga}) \, d\zeta_s \, ,
\end{equation}
where the last integral in~\eqref{d1} has to be interpreted in the rough sense (see Section~\ref{sec:rough} for more details about this notion).  More generally we are interested in rough differential equations of the form
\beq\label{eq:rde-1}
x_t^{\ga} = x_0 + \iot  B(s, x_s^{\ga}, \ga_s) \, ds + \iot \si(s, x_s^{\ga}) \, d\zeta_s,
\eeq
where the coefficients $B$ and $\si$ enjoy convenient continuity properties with respect to their arguments. Note that the coefficient $\si$ in~\eqref{eq:rde-1} does not depend on the control $\ga$, in order to avoid degeneracy problems of the kind mentioned in \cite{diehl,allan-cohen}.
With this framework in hand,
the control problem of interest in this paper consists in maximizing an objective functional over $[0,T]$, for a given time horizon $T>0$. This functional takes the form
\beq\label{eq:cont-1}
	J_T(\gamma) \equiv \int_0^T F(s,x_s^\gamma,\gamma_s) ds + G(x_T^\gamma),
	\quad\text{for}\quad \ga\in\mathcal{K} \, ,
\eeq
where $\mathcal K$ is an appropriate class of \emph{admissible} relaxed control curves. In what follows we assume that the running cost rate $F$ is bounded and continuous in $x$ (uniformly over $\gamma \in \mathcal K$) and Lipschitz in $\gamma$, and that the terminal cost $G$ is bounded and continuous. 

Relaxed optimal controls have (classically) been used to argue the existence of optimal controls both in deterministic settings (see \cite{mcshane,lewis-vintner,roubicek}) and stochastic settings (see \cite{fleming-nisio-a,fleming-nisio-b,elkaroui,mazliak-nourdin} under various assumptions in diffusive settings). Roughly speaking, this classical use mainly stems from the compactification and convexification of the control space afforded by the relaxation. In the deterministic setting, the focus is primarily on systems of ordinary differential equations (ODEs). On the other hand, the stochastic literature primarily handles diffusion processes, and is therefore concerned with Markov relaxed controls. 

Our own inspiration to use relaxed controls in a reinforcement learning context comes from more recent contributions.  Wang et al~\cite{wang-zari-zhou} and Kim and Yang~\cite{KimYang} observed that relaxed controls form a natural framework for modeling exploration strategies for reinforcement learning in continuous-time settings. In particular, if in equation~\eqref{eq:cont-1} one uses the running cost objective
\begin{equation}\label{d2}
F(t,x,\gamma) = e^{-\rho t} \left( \int_U r(x,u) \dot \gamma_s(u) du -\lambda \int_U \dot \gamma_s(u) \log \dot \gamma_s(u) du \right),
\end{equation}
where we have assumed that the measure $\ga_{s}$ admits a density $\dot \gamma_s$,
this forms an `exploratory' framework to model repetitive learning under exploration. Within this landmark, the multiplier $\lambda > 0$ in~\eqref{d2} captures the tradeoff between exploitation (of the running relaxed control cost $\int_U r(x,u) \dot \gamma_s(u) du$) and exploration (the differential entropy $-\int_U \dot \gamma_s(u) \log \dot \gamma_s(u) du $). This is easily seen to emulate the maximum entropy methods used in discrete-time reinforcement learning (see~\cite{Neu} and references therein). Working in an ordinary differential equation setting, Kim and Yang~\cite{KimYang} derive the optimal relaxed control using the corresponding Hamilton-Jacobi-Bellman (HJB) equation. In the end they obtain the so-called \emph{exploratory} HJB (Feynman-Kac) partial differential equation~\cite{tang-zhang-zhou}, in a nonlinear general cost setting. Wang et al~\cite{wang-zari-zhou}, on the other hand, consider the linear-quadratic controlled diffusion setting, and fully characterize the exploratory HJB and optimal control while considering an expected/stochastic cost formulation. 

With this inspiration in mind, our objective is to generalize both~\cite{KimYang} and~\cite{wang-zari-zhou} to a general noisy driving force. This means that, in particular, we do not tackle directly the aforementioned reinforcement learning problem in a rough environment. In contrast, we focus on a preliminary though necessary question: characterization of the optimal value function and  HJB equation for relaxedly controlled rough differential equations of the form~\eqref{eq:rde-1}, in case of a pathwise control problem minimizing objectives of the form~\eqref{eq:cont-1}.
~Specifically, we address the following issues:

\newpage

\begin{enumerate}[wide, labelwidth=!, labelindent=0pt, label=\textnormal{(\roman*)}]
\setlength\itemsep{.1in}
	\item 
	Definition and resolution of an equation of the form~\eqref{d1} under minimal assumptions on the function $s \mapsto b(s,x,a)$. Note that this minimal regularity will be needed when we consider controls which are driven by a rough path. Our main result in this direction is Corollary~\ref{cor:eq-with-gamma} below. 
	
	\item
	Study of the value function $V$ associated to an optimization problem for the cost objective $J_T$ in~\eqref{eq:cont-1}-\eqref{d2}, when the dynamics is given by~\eqref{d1}. This includes the existence of the value function (Proposition~\ref{prop:sup}), as well as regularity properties (Proposition~\ref{prop:V-reg}).
	
	\item
	Proper definition and resolution of a HJB equation, whose unique rough viscosity solution is the value $V$. Informally, this rough PDE takes the following form:
		\[
			\partial_t v(t,y) + \sup_{\gamma \in \mathcal P(U)} \left\{ \nabla v(t,y) \cdot \int b(t,y,a)\gamma(da) + F(t,y,\gamma) \right\} + \nabla v(t,y) \cdot \sigma(t,y)  \dot{\zeta}_t = 0 \, ,
		\]
		with terminal condition $\sigma(t,y) = G(y)$. Observe that $b$ and $\sigma$ are the coefficients appearing in~\eqref{d1}, and that $F$ is the objective function in~\eqref{eq:cont-1}. While this equation is quite natural in our relaxed control context, its rigorous resolution represents a substantial part of our efforts. This encompasses the definition of a set of rough test functions (Definition~\ref{def:j-d}) and rough viscosity solutions (Definition~\ref{def:j-f}). We also show that the value function $V$ solves this HJB equation (Theorem~\ref{thm:v-hjb}). Eventually we prove existence and uniqueness of a solution for the rough HJB equation (Theorem~\ref{thm:ex-uniq}), as well as approximation results (Theorem~\ref{thm:smooth-approx}).
\end{enumerate}

\noindent
As the reader can see, the current paper should be viewed as a theoretical contribution giving a rigorous treatment of relaxed controls of rough differential equations. We believe this to pave the way to a proper treatment of reinforcement learning in general noisy (or rough) environments. In particular,  noisy environments that display rougher and finer path structure than Brownian motion abound in a range of applications. For instance, in both finance and stochastic networks, the control of stochastic dynamics driven by fractional Brownian motion (fBM) are natural. A typical example in the finance literature is the seminal work by \cite{Gatheral}, where in the (log-)volatility is shown to behave essentially like a fBM with Hurst index $< 1/2$. In the stochastic networks literature, heavy-traffic limits of some classes of workload input processes have been shown to have fBM limits with Hurst index $H < 1/2$~\cite{Kurtz} as well as $H \geq 1/2$~\cite{AramanGlynn}. The optimal control of stochastic networks with fBM-type input has been investigated motivated by these insights~\cite{Leao,GhoshWeera,LeeWeera,Dai}. The control of stochastic systems with delay dynamics and full memory present another important class of non-Markovian control problems~\cite{KolShai} (see~\cite{Hoglund} as well).
 However, reinforcement learning (RL) in these types of environments requires the development of new tools, and this paper is a critical contribution to this emerging literature.
 We should thus mention at this point that our main results apply to any continuous stochastic process giving raise to a rough path, with H\"older regularity exponent larger than $1/3$. This obviously include multidimensional  fractional Brownian motions with Hurst parameter $H>1/3$.
 
 We are also strongly motivated by recent developments in the deep learning literature relating deep neural network architectures with control and optimization of differential equation systems. Specifically, recent work by Kidger et al~\cite{Kidger1} and Morrill et al~\cite{Morrill} argue that controlled differential equations can be viewed as `continuum' recurrent neural networks (RNNs) (i.e. in the limit of an infinite number of layers), for learning maps between path spaces. They term these as {\it neural controlled differential equations} (neural CDEs) when the driving signal is deterministic, or as {\it neural stochastic differential equations} (neural SDEs) when the driving signal is stochastic (typically Brownian motion); see~\cite{Kidger2}. Following E et al~\cite{E1,E2} on training neural ODEs, the training problem for neural CDEs/SDEs can potentially be viewed as a mean field optimal control problem.
It is well known that training RNNs is complicated~\cite{Sutskever}, being highly sensitive to initial conditions and the input sample path, and it is anticipated that an optimal control approach may yield more stable training paths. As noted before, the current paper is the first step towards the development of such novel training methodologies for neural CDEs and neural SDEs.

The literature on viscosity solutions for nonlinear rough PDEs is somewhat reduced. If we restrict our attention to HJB type equations, let us mention the seminal article~\cite{lions-souganidis} and the paper~\cite{diehl} on control aspects. One should also include the extension in~\cite{allan-cohen} to settings in which the control also enters the rough integral term. In all those contributions, the solution is defined through a limiting procedure, approximating the noisy input $\zeta$ by a sequence $(\zeta^n)_{n \geq 1}$ of smooth functions. While this point of view makes sense for viscosity solutions (which can be seen as limits of regularized objects), it raises at least two important questions: 
\begin{enumerate}[wide, labelwidth=!, labelindent=0pt, label=\textnormal{(\alph*)}]
\setlength\itemsep{.1in}
	\item 
How should we precisely interpret the limiting rough PDE? 
\item
How can we claim that the solution (provided it exists) is really unique? 
\end{enumerate}
At a methodological level, our paper fully answers these questions. To this aim, we resort to the aforementioned strategy, completely different from~\cite{diehl,allan-cohen}, and based on a rigorous definition of rough test functions and rough viscosity solutions. 
This point of view is also adopted in the overview article~\cite{souganidis-course-cetraro}, where a full theory of rough viscosity solutions is developed. See also 
~\cite{buckdahn} for the special case of an equation driven by a 1-d Brownian motion, as well as~\cite{bhauryal} in an anticipative stochastic calculus framework. Our own rough viscosity setting is valuable in its own right, and applicable to settings like~\cite{buckdahn,bhauryal}. Basically, our notion
of test functions is similar to the class of smooth functions, albeit perturbed by a noisy term involving the rough driving signal $\zeta$ in~\eqref{eq:rde-1}. This notion of a test function generalizes those in the case of relaxed deterministic optimal control, where the class of test functions are precisely smooth functions. Further, we adopt a new approach to proving the uniqueness of the viscosity solution by using rough flows (analogous to the stochastic flows invoked in~\cite{kunita}, and developed e.g in~\cite{bailleul-19}). 

The rest of this paper is organized as follows. In Section~\ref{sec:rough} we briefly recall some rough path notions that will be used throughout. Our contributions start in Section~\ref{sec:rded} where we first consider a `generic' rough differential equation with drift and prove existence of solutions under minimal H\"older reqularity on the drift coefficient. This result extends existing results, since we allow for time dependence of the drift and diffusion coefficients in the equation. Note that these minimal regularity conditions are necessary when we consider equations involving a control, which we do in Section~\ref{sec:appli-pathwise-control} of the paper. Next, we derive the HJB-type rough-PDE equation satisfied by the value function of the relaxed optimal control problem, using the dynamic programming principle, in Section~\ref{sec:hjb}. After showing that the value function is not necessarily a smooth (or classical) solution of the rough-PDE in Sec~\ref{sec:dpp}, we present our results on the characterization, existence and uniqueness of viscosity solution in Section~\ref{sec:visc}. 

\begin{notation}
In the sequel $k$-th order simplexes on $[0,T]$ will be denoted by $\cs_{k}$. Namely 
$\cs_k([0,T]) = \{ (t_1, \ldots, t_k) \in [0,T]^k; t_1 \leq \cdots \leq t_k \}$. The norm of a vector $f$ in a vector space $V$ is written as $\cn [f; V]$. All constants $c$ can vary from line to line. Function spaces with space-time regularity are spelled as $\cac^{a,b}$.
\end{notation}

\section{Rough Path Notions}\label{sec:rough}
The following is a short account of the rough path notions used in this article, mostly taken from \cite{gubinelli}. We review the notion of controlled process as well as their integrals with respect to a rough path. 

\subsection{Increments}\label{sec:increm}
For a vector space $V$ and an integer $k \geq 1$, let $\cac_k(V)$ be the set of functions $g:\mathcal{S}_{k}([0,T]) \to V$ such that $g_{t_1 \cdots t_k} = 0$ whenever $t_i = t_{i+1}$ for some $i \leq {k-1}$. Such a function will be called a $(k-1)$-increment, and we set $\cac_{*}(V) = \cup_{k \geq 1} \cac_{k}(V)$. Then the operator $\der : \cac_{k}(V) \to \cac_{k+1}(V)$ is defined as follows
\begin{equation}\label{eq:def-delta}
{\der g}_{t_1 \cdots t_{k+1}} = \sum_{i = 1}^{k+1} (-1)^{k-i} g_{t_1\cdots \hat{t_i} \cdots t_{k+1}} \, ,
\end{equation} 
where $\hat{t_{i}}$ means that this particular argument is omitted.
In particular, note that for $f \in \cac_1(V)$ and $h \in \cac_2(V)$ we have
\beq\label{eq:der}
\der f_{st} = f_t - f_s, \quad \text{and} \quad \der h_{sut} = h_{st} - h_{su} -h_{ut}. 
\eeq
It is easily verified that $\der \der = 0$ when considered as an operator from $\cac_{k}(V)$ to $\cac_{k+2}(V)$.

\noindent
The size of these $k$-increments are measured by H\"older norms defined in the following manner: for $f \in \cac_{2}(V)$ and $\mu > 0$ let
\begin{equation}\label{eq:norm_1}
\|f\|_{\mu} = \sup_{(s,t) \in \mathcal{S}_2([0,T])} \dfrac{|f_{st}|}{{|t-s|}^{\mu}},
\quad\text{  and  }\quad
\cac_2^{\mu}(V) = \lbrace f \in \cac_2(V); {\|f\|}_{\mu} < \infty \rbrace \, .
\end{equation}
The usual H\"older space $\cac_1^{\mu}(V)$ will then be determined in the following way: for a continuous function $g \in \cac_{1}(V)$, we simply set
\begin{equation*}
{\|g\|}_{\mu} = {\|\der g\|}_{\mu} \, ,
\end{equation*}
and we will say that $g \in \cac_1^{\mu}(V)$ iff ${\|g\|}_{\mu}$ is finite. 
\begin{remark} 
	Notice that ${\|\cdot\|}_{\mu}$ is only a semi-norm on $\cac_1(V)$, but we will generally work on spaces for which the initial value of the function is fixed.
\end{remark}

We shall also need to measure the regularity of increments in $\cac_3(V)$. To this aim, similarly to \eqref{eq:norm_1}, we introduce the following norm for $h \in \cac_3(V)$:
\begin{equation}\label{eq:norm_2}
{\|h\|}_{\mu} = \sup_{(s,u,t) \in \mathcal{S}_3([0,T])} \dfrac{|h_{sut}|}{{|t-s|}^{\mu}}.
\end{equation}
Then the $\mu$-H\"older continuous increments in $\cac_3(V)$ are defined as:
\begin{equation*}
\cac_3^{\mu}(V) := \lbrace h \in \cac_{3}(V); {\|h\|}_{\mu} < \infty \rbrace .
\end{equation*}
{Notice that the ratio in \eqref{eq:norm_2} could have been written as $\frac{|h_{sut}|}{|t-u|^{\mu_{1}} |u-s|^{\mu_{2}}}$ with $\mu_{1}+\mu_{2}=\mu$, in order to stress the dependence on $u$ of our increment $h$. However, expression \eqref{eq:norm_2} is simpler and captures the regularities we need, since we are working on the simplex $\cs_{3}$.}

The building block of the {rough path} theory is the so-called sewing map lemma. We recall this fundamental result here.\begin{proposition}\label{prop:La}
	Let $h \in \cac_3^{\mu}(V)$ for $\mu > 1$ be such that $\der h = 0$. Then there exists a unique $g = \Lambda(h) \in \cac_2^{\mu}(V)$ such that $\der g = h$. Furthermore for such an $h$, the following relations hold true:
	\begin{equation*}
	\der \Lambda(h) = h ,
	\quad\text{  and  }\quad
{\|\Lambda h\|}_{\mu} \leq \dfrac{1}{2^{\mu} - 2} {\|h\|}_{\mu} .
	\end{equation*}
\end{proposition}

For our considerations below, we shall need some basic rules about products of increments. Towards this aim, let us specialize our setting to the state space $V=\R$. We will also write $\mathcal{C}_{k}^{\mu}$ for $\mathcal{C}_{k}^{\mu}(\R)$. Then $\left(\mathcal{C}_{\ast}, \der\right)$ can be endowed with the following product: for $g \in \mathcal{C}_n$ and $h \in \mathcal{C}_m$ we let $gh$ be the element of $\mathcal{C}_{m+n-1}$ defined by
\begin{equation*}
(gh)_{t_1, \ldots, t_{m+n-1}} = g_{t_1, \cdots, t_n}h_{t_n,\cdots t_{m+n-1}},
\quad\text{  for  }\quad
(t_1, \ldots, t_{m+n-1}) \in \mathcal{S}_{m+n-1}([0,T]).
\end{equation*}
We now label a rule for discrete differentiation of products, which will be used throughout the article. Its proof is an elementary application of the definition \eqref{eq:def-delta}, and is omitted for sake of conciseness.
\begin{proposition}\label{prop:der_rules}
	The following rule holds true:
	Let $g \in \cac_1$ and $h \in \cac_2$. Then $gh \in \cac_2$ and
	\begin{equation*}
	\der(gh) = -\der g\, h + g\, \der h .
	\end{equation*}
\end{proposition}

\noindent
The iterated integrals of smooth functions on $[0,T]$ are particular cases of elements of $\cac_2$, which will be of interest. Specifically, for smooth real-valued functions $f$ and $g$, let us denote $\int f dg$ by $\mathcal{I}(f dg)$ and consider it as an element of $\cac_2$: for $(s,t) \in \mathcal{S}_{2}\left([0,T]\right)$ we set
\begin{equation*}
\mathcal{I}_{st} (f dg) = \left(\int f dg\right)_{st} = \int_{s}^{t} f_u dg_u .
\end{equation*}

\subsection{Weakly controlled processes}\label{sec:weak-contr}

From now on we will focus on a space $V$ of the form $V = \R^d$. The main assumption on the $\R^d$-valued noise $\bzeta$ of equation~\eqref{eq:de} is that it gives rise to a geometric rough path. This assumption can be summarized as follows.
\begin{hypothesis}\label{hyp:zeta}
	The path $\zeta:[0,T] \to \R^d$ belongs to the H\"older space ${\cac}^{\al}([0,T];\R^d)$ with $\al > \frac{1}{3}$ and $\zeta_0 = 0$. In addition $\zeta$ admits a L\'evy area above itself, that is, there exists a two index map $\mathbf{\zeta}^2 : {\mathcal{S}_{2}\left([0,T]\right)} \to \R^{d,d}$ which belongs to $\cac_{2}^{2\al}(\R^{d,d})$ and such that 
	\beq\label{eq:zeta^2}
	\der {\mathbf{\zeta}}_{sut}^{2;ij} = {\der \zeta}_{su}^{i} \otimes {\der \zeta}_{ut}^{j},
	\quad\text{  and  }\quad
	{\mathbf{\zeta}_{st}^{2;ij} + \mathbf{\zeta}_{st}^{2;ji}} =  \der \zeta_{st}^{i} \otimes \der \zeta_{st}^{j} .
	\eeq
	The $\al$-H\"older norm of $\zeta$ is denoted by:
	\beq\label{eq:zeta-norm}
	\|\mathbf{\zeta}\|_{\al} = \cn(\zeta;\cac_1^{\al}([0,T],\R^d))
	+\cn(\mathbf{\zeta}^2;\cac_2^{2\al}([0,T],\R^{d,d})).
	\eeq
\end{hypothesis}

\begin{remark}
In this article we restrict our analysis to $\al$-H\"older continuous signals with $\al>1/3$. This assumption is imposed in order to limit Taylor type expansions to reasonable lengths. However, generalizations to lower regularities can be achieved thanks to lengthy additional computations.
\end{remark}

\begin{remark}
As mentioned in the introduction, a typical example of path $\zeta$ satisfying Hypothesis~\ref{hyp:zeta} are generic sample trajectories of a fractional Brownian motion $B=(B^{1,H},\ldots,$ $B^{d,H})$ defined on a complete probability space $(\Omega,\cf,\bp)$. In order to ensure that the signal is $\al$-H\"older continuous with $\al>1/3$, we need to restrict the Hurst parameter to be such that $H>1/3$ (see e.g \cite[Chapter 15]{friz-victoir}). All the results stated below apply to this context.
\end{remark}

\noindent
We now define the notion of a weakly controlled process.
\begin{definition}\label{def:weakly-ctrld}
	Let $z$ be a process in $\cac_{1}^{\ka}(\R^n)$ with $\ka \leq \al$ and $2\ka +\ga > 1$. We say that $z$ is weakly controlled by $\zeta$ if $\der z \in \cac_2^{\ka}(\R^n)$ can be decomposed into
	\begin{equation}\label{eq:weakly-controlled}
	\der z^{i} = z^{\zeta;i i_1} \der \zeta^{i_1} + \rho^{i}, 
	\quad\text{   i.e.  }\quad
{\der z}_{st}^{i} = z_{s}^{\zeta;i i_1}{\der \zeta}_{st}^{i_1} + \rho_{st}^{i} ,
	\end{equation}
	for all $(s,t) \in \mathcal{S}_2 \left([0,T]\right)$ and $i=1, \ldots,n$. In the previous formula we assume $z^{\zeta} \in \cac_{1}^{\ka}(\R^{n,d})$ and $\rho$ is a more regular remainder such that $\rho \in \cac_{2}^{2\ka}(\R^{n})$. The space of weakly controlled paths on $[a,b]$ with H\"older continuity $\ka$ and driven by $\zeta$ will be denoted by $\mathcal{Q}_{\zeta}^{\ka}(\R^n)$ and a process $z \in \mathcal{Q}_{\zeta}^{\ka}(\R^n)$ can be considered as a couple $(z, z^\zeta)$. The natural semi-norm on $\mathcal{Q}_{\zeta}^{\ka}(\R^n)$ is given by
	\beq\label{eq:norm-controlled}
	\cn[z;\mathcal{Q}_{\zeta}^{\ka}(\R^n)] = \cn[z; \cac_1^{\ka}(\R^n)] + \cn[z^\zeta; \cac_1^{\infty}(\R^{n,d})] + \cn[z^\zeta; \cac_1^{\ka}(\R^{n,d})] + \cn[\rho; \cac_2^{\ka}(\R^n)].
	\eeq
\end{definition}

\begin{remark}\label{rem:j-e}
In the sequel we will also need to introduce spaces of the form $\cq_{\zeta}^{\ka}({[0,T]}; V)$, where $V$ is an infinite dimensional functional space. As an example, our notion of rough viscosity solution in Section~\ref{sec:hjb} will require $V = \cac^1(\R^m)$.
\end{remark}

In the following proposition we give an elaboration of the classical composition rule in~\cite{gubinelli}, for controlled processes taking into account time dependent functions. 
 
\begin{proposition}\label{prop:smooth_of_weak}
Let $z$ be a controlled process in $\cq_{\zeta}^{\ka}([a,b];\R^n)$, with decomposition \eqref{eq:weakly-controlled} and initial condition $z_0$. Let also $L$ be a function in $\cac_b^{\tau, 2} ([a,b] \times \R^n;\R^{m})$ for $\tau \geq 2\ka$. Set $\hat{z}_t = L(t, z_t)$, with initial condition $\hat{\al} = L(a, z_0)$. Then the increments of $\hat{z}$ can be decomposed into
$$
\der \hat{z}_{st} = \hat{z}_s^{\bzeta} \, \der \zeta_{st} + \hat{\rho}_{st} + (\der L(\cdot, z))_{st},
$$
with the convention that $\nabla_z L(s, z)\in\R^{m,n}$, and $\hat{z}^{\zeta}, \hat{\rho}$ defined by
\beq\label{eq:6-terms}
\hat{z}_s^{\zeta} = \nabla_z L(s, z_s) z_s^{\zeta} \quad \text{and} \quad \hat{\rho}_{st}=\nabla_z L(s,z_s) \rho_{st} +      \lc  {\lp \der L(\cdot ,z) \rp}_{st} -\nabla_z L(s,z_s) \der z_{st} \rc . 
\eeq
Furthermore $\hat{z} \in \cq_{\zeta}^{\ka} ([a,b]; \R^m)$, and
\beq\label{eq:6-a}
\cn \lc \hat{z}; \cq_{\zeta}^{\ka}([a,b];\R^m) \rc \leq c_L \lp 1 + \cn^2\lc z; \cq_{\zeta}^{\ka} ([a,b]; \R^n) \rc \rp.
\eeq
\end{proposition}
\begin{proof}
We give a brief proof for sake of completeness. Recall that $\cn[\hat{z}; \cq_{\bzeta}^{\ka}]$ is given by
\beq\label{eq:6-0}
\cn[\hat{z}; \cq_{\bzeta}^{\ka}] = \cn[\hat{z}; \cac_1^{\ka}] + \cn[\hat{z}^{\bzeta}; \cac_1^{\infty}] + \cn[\hat{z}^{\bzeta}; \cac_1^{\ka}] + \cn[\hat{\rho}; \cac_2^{2\ka}].
\eeq
In the sequel we will just show how to bound the term $\cn[\hat{z}; \cac_1^{\ka}]$, the other terms being handed similarly to \cite[Proposition~4]{gubinelli}. 

Recall that $\hat{z}_t = L(t, z_t)$. Hence we have
$$
\lln \der \hat{z}_{st} \rrn \leq \lln L(t, z_t) - L(s, z_t) \rrn + \lln L(s, z_t) - L(s, z_s) \rrn,
$$
from which we easily get 
$$
\lln \der \hat{z}_{st} \rrn \leq {\|L\|}_{\cac_b^{\tau, 2}} {|t-s|}^{\tau} + {\|\nabla_z L\|}_{\infty} \cn [z; \cac_1^{\ka}] {|t-s|}^{\ka}.
$$
Therefore we end up with 
$$
\cn[\hat{z}; \cac_1^{\ka}] \leq {\|L\|}_{\cac_b^{\tau, 2}} T^{\tau - \ka} + {\|\nabla_z L\|}_{\infty} \cn[z; \cac_1^{\ka}].
$$
This bound is consistent with \eqref{eq:6-a}. 
As mentioned above, the rest of the terms in \eqref{eq:6-0} can be bounded similarly to \cite[Proposition~4]{gubinelli} in order to obtain the required relation~\eqref{eq:6-a}. 
\end{proof}

One of the main reasons to introduce weakly controlled paths is that they provide a natural set of functions which can be integrated with respect to a rough path. Below we handle mixed integrals with a rough term and a Lebesgue type term, for which we need an additional notation.
\begin{notation}\label{not:leb-int}
Let $\eta$ be a path in $\cac_b^0([a,b];\R)$, where $\cac_b^0([a,b];\R)$ stands for the space of $\R$-valued continuous functions defined on $[a,b]$. For all $a \leq s \leq t \leq b$ we denote by $\ci_{st}(\eta)$ the Lebesgue integral $\int_s^t \eta_u du$.
\end{notation}
We are now ready to state our integration result. Notice again that our formulation of the integral differs from the standard one, due to the fact that we are paying extra attention to the drift term in our control setting.
\begin{proposition}\label{prop:integral_as_weak}
For a given $\al > \frac{1}{3}$ and $\ka < \al$, let $\zeta$ be a process satisfying Hypothesis~\ref{hyp:zeta}. Furthermore, let $\eta \in \cac_b^0([a,b]; \R)$ as in Notation~\ref{not:leb-int}, and $\mu \in \cq_{\zeta}^{\ka}([a,b]; \R^d)$ as given in Definition~\ref{def:weakly-ctrld} whose increments can be decomposed as
\beq\label{b1}
\der \mu^i = \mu^{\zeta;i i_1} \der \zeta^{i_1} + \rho^{\mu;i}, \quad\text{ where } \mu^{\zeta} \in \cac_1^{\ka}([a,b]; \R^{d,d}), \, \rho^{\mu} \in \cac_2^{2\ka}([a,b]; \R^d).
\eeq
Define $z$ by $z_a = y \in \R$ and 
\beq\label{eq:b}
\der z = \mu^i\der \zeta^i + \mu^{\zeta;i i_1} \mathbf{\zeta}^{2;i_1 i} + \ci(\eta) + \Lambda(\rho^{\mu;i} \der \zeta^i + \der \mu^{\zeta;i i_1} \mathbf{\zeta}^{2; i_1 i}),
\eeq
where the term $\ci(\eta)$ is introduced in Notation~\ref{not:leb-int} and $\laa$ is the sewing map in Proposition~\ref{prop:La}. Then:
\begin{enumerate}[wide, labelwidth=!, labelindent=0pt, label=\textnormal{(\arabic*)}]
\setlength\itemsep{.05in}

\item 
 The path $z$ is well-defined as an element of $\cq_{\zeta}^{\ka}([a,b]; \R)$, and $\der z_{st}$ coincides with the Lebesgue-Stieltjes integral $\int_s^t \mu_u^i d\zeta_u^i + \int_s^t \eta_u du$  whenever $\zeta$ is a differentiable function.

\item The semi-norm of $z$ in $\cq_{\zeta}^{\ka} ([a,b]; \R)$ can be estimated as 
\begin{multline}\label{eq:6-b}
\cn[z; \cq_{\zeta}^{\ka}([a,b]; \R)] \\
\leq c_{\zeta} \lcl 1 + \|\mu_a\| + (b-a)^{\al-\ka} \lp \|\mu_a\| + \cn[\mu; \cq_{\zeta}^{\ka}([a,b]; \R^d)] + \cn[\eta; \cac_b^0([a,b]; \R)] \rp \rcl,
\end{multline}
where the constant $c_{\zeta}$ verifies $c_{\zeta} \leq c(|\zeta|_{\al} + |\mathbf{\zeta}^2|_{2\al})$ for a universal constant $c$.

\item It holds that 
\beq\label{eq:riemann}
\der z_{st} = \lim_{|\Pi_{st}| \to 0} \sum_{q=0}^n \lc \mu_{t_q}^i \der \zeta_{t_q t_{q+1}}^i + \eta_{t_q} (t_{q+1}-t_q) + \mu_{t_q}^{\zeta;i i_1} \mathbf{\zeta}_{t_q, t_{q+1}}^{2; i_1 i} \rc,
\eeq
for any $a \leq s \leq t \leq b$. In \eqref{eq:riemann} the limit is taken over all partitions $\Pi_{st} = \{s=t_0, \ldots, t_n=t\}$ of $[s,t]$, as the mesh of the partition goes to zero. 
\end{enumerate}
\end{proposition}
\begin{proof}
We briefly give a proof of relation~\eqref{eq:6-b} as the rest is contained in \cite[Theorem 1]{gubinelli}. Recall that 
\beq\label{eq:z-norm}
\cn [z ; \cq_{\bzeta}^\ka] = \cn[z; \cac_1^{\ka}] + \cn[\mu; \cac_1^{\infty}] + \cn[\mu; \cac_1^{\ka}] + \cn[\rho^{\ast}; \cac_2^{2\ka}],
\eeq
where 
\beq\label{eq:rho*}
\rho^{\ast} =  \mu^{\zeta} \mathbf{\zeta}^{2} + \ci(\eta) + \Lambda(\rho \der \zeta + \der \mu^{\zeta} \mathbf{\zeta}^{2}).
\eeq
Let us first bound the term $\rho^{\ast}$ above. By Proposition~\ref{prop:La} we have that
\beq\label{eq:prop-6-a}
\| \Lambda(\rho \der \zeta + \der \mu^{\zeta} \mathbf{\zeta}^{2})\|_{3 \ka} 
\lesssim \|\rho \der \bzeta + \der \mu^{\bzeta} \bzeta^2\|_{3\ka} 
\leq (b-a)^{\al - \ka} \cn[\rho; \cac_2^{2\ka}] \|\bzeta\|_{\al} + (b-a)^{2\al- 2\ka} \cn[\mu; \cq_{\bzeta}^{\ka}]\|\bzeta\|_{\al},
\eeq
while it is readily checked that
\beq\label{eq:prop-6-b}
\cn[\mu^{\bzeta} \bzeta^2 + \ci(\eta); \cac_2^{2\ka}] \leq (b-a)^{2\al - 2\ka} \lp \cn[\mu; \cq_{\bzeta}^{\ka}] \|\bzeta\|_{\al} + \cn[\eta; \cac_b^0] \rp.
\eeq
Plugging \eqref{eq:prop-6-a} and \eqref{eq:prop-6-b} into \eqref{eq:rho*}, we end up with
\begin{eqnarray}\label{eq:rho*-b}
\cn[\rho^{\ast}; \cac_2^{2\ka}] 
&\leq& 
c_3 (b-a)^{\al - \ka} \lp \cn[\mu; \cq_{\bzeta}^{\ka}] + \cn[\rho; \cac_2^{2\ka}] + \cn[\eta; \cac_b^0] \rp \notag\\
&\leq& 
c_3 (b-a)^{\al - \ka} \lp \cn[\mu; \cq_{\bzeta}^{\ka}] + \cn[\eta; \cac_b^0] \rp.
\end{eqnarray}
In order to bound $\cn[\mu; \cac_1^{\ka}]$ in \eqref{eq:z-norm} we could use the trivial bound $\cn[\mu; \cac_1^{\ka}] \leq \cn[\mu; \cq_{\bzeta}^{\ka}]$, which stems directly from our definition~\eqref{eq:norm-controlled}. Since we wish to extract some power of $(b-a)$ in our upper bound, we will follow a different route. Namely we write
\begin{multline}\label{eq:prop-6-c}
\cn[\mu; \cac_1^{\ka}] = \sup_{s,t \in [a,b]} \dfrac{|\mu_s^{\bzeta} \der \bzeta_{st} + \rho_{st}|}{|s-t|^{\ka}}\leq \sup_{s,t \in [a,b]} \lp \dfrac{|\mu_s^{\bzeta}| |\der \bzeta_{st}|}{|s-t|^{\ka}} + \dfrac{|\rho_{st}|}{|s-t|^{\ka}} \rp\\
 \leq \cn[\mu; \cac_1^{\infty}](b-a)^{\al - \ka} \|\bzeta\|_{\al} + (b-a)^{\ka} \cn[\rho;\cac_2^{2\ka}].
\end{multline}
Combining \eqref{eq:rho*-b}-\eqref{eq:prop-6-c}, and similarly obtainable bounds for $\cn[z;\cac_1^{\ka}]$ and $\cn[\mu; \cac_1^{\ka}]$ we obtain the relation \eqref{eq:6-b}.
\end{proof}

\subsection{Strongly controlled processes}

In the sequel, some of our computations will require a finer description than~\eqref{eq:weakly-controlled} for controlled processes. In this section we introduce this kind of decomposition and we will also derive an application to rough integration.
Let us start with the definition of strongly controlled process. 
\begin{definition}\label{def:str-contr}
Similarly to Definition~\ref{def:weakly-ctrld}, consider a process $\nu \in \cac_1^{\ka}(\R^n)$ with $\ka \leq \al$ and $2\ka+\al > 1$. Let us suppose that the increment $\der z$ lies in $\cac_2^{\ka}$ and can be decomposed as 
\beq\label{eq:str-contr-a}
\der \nu^j = \nu^{\zeta;j i_1} \der \zeta^{i_1} + \nu^{\bzeta^{2}; j i_1 i_2} \bzeta^{2; i_2 i_1} + \rho^{\nu;j},
\eeq
where the regularity for the paths $\nu^{\zeta}$, $\nu^{\bzeta^2}$ and $\rho^{\nu}$ are respectively
\beq\label{eq:str-contr-b}
\nu^{\zeta}, \nu^{\bzeta^2} \in \cac_1^{\ka}([a,b]), \quad\text{and}\quad \rho^{\nu} \in \cac_2^{3\ka}([a,b]).
\eeq
In addition, we assume that $\nu^{\zeta}$ is also a weakly controlled process as in Definition~\ref{def:weakly-ctrld}, with decomposition
\beq\label{eq:str-contr-l}
\der \nu^{\zeta; ji_1} = \nu^{\zeta^{2}; j i_1 i_2} \der \zeta^{i_2} + \rho^{\nu^{\zeta};j},
\eeq
where $\nu^{\zeta^{2}; j i_1 i_2}$ is like in \eqref{eq:str-contr-a} and $\rho^{\nu^{\zeta}; j}$ is a remainder in $\cac_2^{2\ka}([a,b])$. Then we say that $\nu$ is a strongly controlled process in $[a,b]$. We call $\tilde{\cq}_{\zeta}^{\ka}(\R^n)$ this set of paths, equipped with the semi-norm
$$
\cn[\nu; \tilde{\cq}_{\zeta}^{\ka}(\R^n)] = \cn[\nu;\cac_1^{\ka}] + \cn[\nu^{\zeta};\cac_1^{\infty}] + \cn[\nu^{\zeta^{2}}; \cq_{\zeta}^{\ka}] + \cn[\rho;\cac_2^{3\ka}].
$$
\end{definition}

On our way to the definition of viscosity solutions, we will encounter integrals of controlled processes with respect to other controlled processes. We recall how to define this kind of integral in the proposition below, which is a slight elaboration of \cite[Theorem~1]{gubinelli} and requires the introduction of strongly controlled processes. 

\begin{proposition}\label{prop:integral_controlled}
For a given $\al > \frac{1}{3}$ and $\ka < \al$, let $\zeta$ be a process satisfying Hypothesis~\ref{hyp:zeta}. 
Given $m,n \geq 1$, we consider a weakly controlled process $\mu \in \cq_{\bzeta}^{\ka}([a,b]; \R^m)$ as introduced in Definition~\ref{def:weakly-ctrld} and a strongly controlled process $\nu \in \tilde{\cq}_{\zeta}^{\ka}([a,b]; \R^n)$ as given in Definition~\ref{def:str-contr}. Specifically, for every $i=1, \ldots , m$ and $j=1, \ldots, n$ we assume the decomposition
\begin{align}
\der \mu^i &= \mu^{\zeta;i i_1} \der \zeta^{i_1} + \rho^{\mu; i}, \label{eq:int_contr-g} \\ 
\der \nu^j &= \nu^{\zeta;j i_1} \der \zeta^{i_1} + \nu^{\bzeta^2; j i_1 i_2}\bzeta^{2; i_2 i_1} +  \rho^{\nu;j}, \label{eq:int_contr-h}
\end{align}
where the regularities for the paths $\mu^{\bzeta}, \mu^{\bzeta^2}, \rho^{\mu}, \nu^{\bzeta}, \nu^{\bzeta^2}, \rho^{\nu}$ are respectively
\begin{equation*}
\mu^{\bzeta; i i_1}, \nu^{\bzeta; j i_1}, \nu^{\bzeta^2; j i_1 i_2} \in \cac_1^{\ka}([a,b]), \qquad
\rho^{\mu; i} \in \cac_2^{2 \ka} ([a,b]), \qquad
\rho^{\nu; j} \in \cac_2^{3 \ka} ([a,b]).
\end{equation*}
Furthermore, let $\eta \in \cac_1^0([a,b]; \R^{m, n})$. For $1 \leq i \leq m$ and $1\leq j \leq n$ we define a controlled process $z^{ij}$ by setting $z_a^{ij} = y^{ij} \in \R$ and a decomposition
\beq\label{eq:int_contr-e}
\der z_{st}^{ij} = \cg_{st}^{ij} + \ci_{st}(\eta^{ij}) + \laa_{st}\lp \cj^{ij} \rp,
\eeq
where the term $\ci(\eta^{ij})$ is introduced in Notation~\ref{not:leb-int}, where the increment $\cg^{ij}$ is defined by 
\beq\label{eq:int_contr-i}
\cg_{st}^{ij} = \mu_s^i \nu_s^{\bzeta;j i_1} \der \bzeta_{st}^{i_1} + \lp \mu_s^i \nu_s^{\bzeta^2; j i_1 i_2} + \mu_s^{\bzeta; i i_2} \nu_s^{\bzeta; j i_1} \rp \bzeta_{st}^{2; i_2 i_1},
\eeq
and where $\cj^{i j}$ is an increment in $\cac_3^{3\ka}$ satisfying
\beq\label{eq:int_contr-f}
\cj_{sut}^{i j} = - \der \cg_{sut}^{ij}, 
\quad \text{and} \quad
\cn[\cj^{ij}; \cac_3^{3\ka}([a,b])] \leq c_{\bzeta} \cn[\mu^i; \cq_{\zeta}^{\ka}([a,b])] \cn[\nu^j; \tilde{\cq}_{\zeta}^{\ka}([a,b])].
\eeq
Then similarly to Proposition~\ref{prop:integral_as_weak}, the following assertions hold true:
\begin{enumerate}[wide, labelwidth=!, labelindent=0pt, label=\textnormal{(\arabic*)}]
\setlength\itemsep{.02in}

\item 
The path $z$ is well-defined as an element of $\cq_{\zeta}^{\ka}([a,b]; \R^{m,n})$, and $\der z_{st}$ coincides with the Lebesgue-Stieltjes integral $\int_s^t \mu_u \otimes d\nu_u + \int_s^t \eta_u du$  whenever $\zeta$ is a differentiable function.


\item It holds that 
\begin{equation}\label{eq:riemann-2}
\der z_{st}^{ij} = \lim_{|\Pi_{st}| \to 0} \sum_{q=0}^n \lc \mu_{t_q}^i  \nu_{t_q}^{\bzeta;j i_1} \der \bzeta_{t_q t_{q+1}}^{i_1} 
+ \lp \mu_{t_q}^i \nu_{t_q}^{\bzeta^2;j i_1 i_2} 
+ \mu_{t_q}^{\bzeta; i i_1} \nu_{t_q}^{\bzeta; j i_2} \rp \bzeta_{t_q t_{q+1}}^{2; i_2 i_1} +\eta_{t_q}^{ij} (t_{q+1} - t_q) \rc \, ,
\end{equation}
for any $a \leq s \leq t \leq b$. In \eqref{eq:riemann-2} the limit is taken over all partitions $\Pi_{st} = \{s=t_0, \ldots, t_n=t\}$ of $[s,t]$, as the mesh of the partition goes to zero. 
\end{enumerate}

\begin{proof}
For sake of completeness, we shall only prove relations \eqref{eq:int_contr-e} - \eqref{eq:int_contr-f} for smooth processes. The remaining arguments are standard considerations in rough path theory.

Let us then consider smooth paths $\mu, \nu, \zeta$ with smooth decompositions  \eqref{eq:int_contr-g}-\eqref{eq:int_contr-h}. Then $\int_s^t \mu_u^i d \nu_u^j$ can be defined as a Riemann-Stieltjes integral. Let us simply write 
$$
\int_s^t \mu_r^i d \nu_r^j = \mu_s^i \der \nu_{st}^j + \int_s^t \der \mu_{sr}^i d \nu_r^j.
$$
Then expand $\mu$ according to \eqref{eq:int_contr-g} and $\nu$ according to \eqref{eq:int_contr-h} after an integration by parts procedure (of the form $\int_s^t \der \mu_{sr} d\nu_r = \der \mu_{st} \der\nu_{st} - \int_s^t \der \nu_{sr} d\mu_r$). We let the reader go through the tedious algebraic details and check that we end up with a relation of the form
\begin{equation}\label{b2}
\int_s^t r_{sr}^{\mu, \nu;ij} \cdot d \zeta_r = \int_s^t \mu_r^i d\nu_r^j - \cg_{st}^{ij},
\end{equation}
where $\cg^{ij}$ is defined by \eqref{eq:int_contr-i}. Moreover, it is readily checked that $\der q_{sut}^{ij} = 0$, where $q_{st}^{ij}$ is the integral $\int_s^t \mu_r^i d \nu_r^j$.
Hence applying $\delta$ on  both sides of~\eqref{b2} we obtain 
$$
\der \lp \int r_{sr}^{\nu, \mu;ij} \cdot d \zeta_r \rp = \der \cg^{ij}.
$$
In order to define $\int \mu^i d \nu^j$ through the sewing map recalled in Proposition~\ref{prop:La}, it thus remains to show that $\der \cg^{ij} \in \cac_3^{3 \ka}([a,b])$. Let us now perform the algebraic computations leading to an expression for $\der \cg^{ij}$. First we apply the elementary Proposition~\ref{prop:der_rules} as well as our assumption~\eqref{eq:zeta^2} to the right hand side of \eqref{eq:int_contr-i}. This yields
\beq\label{eq:int_contr-k}
\der \cg_{sut}^{ij} = - \der \mu_{su}^i \nu_u^{\zeta; j i_1}
 \der \zeta_{ut}^{i_1} - \mu_s^i \der \nu_{su}^{\zeta; ji_1} \der \zeta_{ut}^{i_1} +
 (\mu_s^i \nu_s^{\zeta^2;j i_1 i_2} + \mu_s^{\zeta; i i_1} \nu_s^{\zeta; j i_2}) \der \zeta_{su}^{i_2} \der \zeta_{ut}^{i_1} + R_{sut}^{\cg; ij},
\eeq
where the remainder $R^{\cg}$ is defined as below and is easily seen as an element of $\cac_3^{3\ka}([a,b])$:
$$
R_{sut}^{\cg;ij} = \der {\lp \mu^i \nu^{\zeta^2; j i_1 i_2}  + \mu^{\zeta; i_1 i_2} \nu^{\zeta; j i_1}\rp}_{su} \der \zeta_{ut}^{2; i_2 i_1}.
$$
Next in \eqref{eq:int_contr-k} we use the decomposition \eqref{eq:int_contr-g} for $\der \mu^i$. We also invoke the fact that $\nu$ is a strongly controlled process, which implies that $\nu^{\zeta}$ admits the decomposition \eqref{eq:str-contr-l}. Some elementary manipulations then reveal that \eqref{eq:int_contr-k} can be simplified as
\beq\label{eq:int_contr-m}
\der \cg_{sut}^{ij} = \rho_{su}^{\mu;i} \nu_u^{\zeta; ji_1} \der \zeta_{ut}^{i_1} + \mu_s^i {\rho}_{su}^{\nu^{\zeta;ji_1}} \der \zeta_{ut}^{i_1} + R_{sut}^{\cg; ij}.
\eeq
It can be directly observed that $\der \cg^{ij} \in \cac_3^{3\ka}$ from expansion \eqref{eq:int_contr-m}. Hence a direct application of Proposition~\ref{prop:La} finishes our proof.
\end{proof}
\end{proposition}

We close this section by spelling out a technical result which will be used for our HJB equations. We shall thus consider controlled processes which are indexed by an additional spatial parameter $\theta$. We give a composition rule for such processes. 

\begin{proposition}\label{prop:str_comp}
For a given $\al > \frac{1}{3}$ and $\ka < \al$, let $\bzeta$ be a process satisfying Hypothesis~\ref{hyp:zeta}. Consider two $\cac^3(\R^m;\R^m)-$valued strongly controlled processes $\mu$ and $\nu$ admitting a decomposition of type~\eqref{eq:str-contr-a}:
\begin{align}\label{eq:mu_nu_str}
\der \mu^j(\theta) &= \mu^{\bzeta; j i_1}(\theta) \der \bzeta^{i_1} + \mu^{\bzeta^2; ji_1 i_2}(\theta) \bzeta^{2;i_2 i_1} + \rho^{\mu;j}(\theta) \nonumber\\
\der \nu^j(\theta) &= \nu^{\bzeta;ji_1}(\theta) \der \bzeta^{i_1} + \nu^{\bzeta^2; ji_1 i_2}(\theta) \bzeta^{2; i_2 i_1} + \rho^{\nu; j}(\theta),
\end{align}
for $\theta \in \R^m$. Then $\mu \circ \nu$ is another strongly controlled process with decomposition 
\begin{align}
{\lc \mu \circ \nu \rc}^{\bzeta; j i_1}(\theta) &= \mu^{\bzeta; j i_1}(\nu(\theta)) + \pt_{\eta^k} \mu^j(\nu(\theta)) \nu^{\bzeta; ki_1}(\theta) \label{eq:mu_nu_str_decomp-i}\\
{\lc \mu \circ \nu \rc}^{\bzeta^2; j i_1 i_2} (\theta) &= \mu^{\bzeta^2; j i_1 i_2}(\nu(\theta)) + \pt_{\eta^k} \mu^j(\nu(\theta)) \nu^{\bzeta^2; k i_1 i_2}(\theta) + \pt_{\eta^k} \mu^{\bzeta; j i_2}(\nu(\theta)) \nu^{\bzeta; k i_1}(\theta)\nonumber \\ 
&+ \pt_{\eta^k} \mu^{\bzeta; j i_1}(\nu(\theta)) \nu^{\bzeta; k i_2}(\theta) 
+\pt_{\eta^{k_1} \eta^{k_2}}^2 \mu^j (\nu(\theta)) \nu^{\bzeta; k_1 i_1}(\theta) \nu^{\bzeta; k_2 i_2}(\theta) \label{eq:mu_nu_str_decomp-ii}.
\end{align}
\end{proposition}

\begin{proof}
Let us first decompose the increment $\der [\mu \circ \nu(\theta)]_{st}$ as follows:
$$
\der [\mu \circ \nu](\theta)_{st} = \ca_{st}^1(\theta) + \ca_{st}^2(\theta),
$$
where
$$
\ca_{st}^1(\theta) = \mu_t(\nu_t(\theta)) - \mu_t(\nu_s(\theta)), \quad \text{and} \quad
\ca_{st}^2(\theta) = \mu_t(\nu_s(\theta)) - \mu_s(\nu_s(\theta)).
$$
The next step is to derive the strongly-controlled representations for $\ca^1$ and $\ca^2$ separately. To deal with the former term we have the following Taylor expansion using the fact that $\mu$ is thrice differentiable in space:
\begin{align}\label{eq:ca^1-i}
\ca_{st}^{1;j}(\theta) &= \mu_t^j(\nu_t(\theta)) - \mu_t^j(\nu_s(\theta)) \nonumber \\
&=  \pt_{\eta^k} \mu_t^j(\nu_s(\theta)) \der \nu_{st}^k(\theta) + \dfrac{1}{2} \pt_{\eta^{k_1} \eta^{k_2}}^2 \mu_t^j(\nu_s(\theta)) \der \nu_{st}^{k_1}(\theta) \der \nu_{st}^{k_2}(\theta) + R_{st},
\end{align}
where here and in the following $R$ is a generic $\cac^{\mu}$ remainder term for $\mu > 1$. Using the strongly controlled process representation for $\nu$ in \eqref{eq:mu_nu_str} we now have from \eqref{eq:ca^1-i} that
\beq\label{eq:ca^1-ii}
\ca_{st}^{1;j}(\theta) = \pt_{\eta^k} \mu_t^j(\nu_s(\theta)) F_{st}^{k} 
+ \dfrac{1}{2} \pt_{\eta^{k_1} \eta^{k_2}} \mu_t^j(\nu_s(\theta)) F_{st}^{k_1} F_{st}^{k_2} + R_{st},
\eeq
where 
$$
F_{st}^{k} = \nu_s^{\bzeta; k i_1}(\theta) \der \bzeta_{st}^{i_1} + \nu_s^{\bzeta^2; k i_1 i_2} (\theta) \bzeta_{st}^{2; i_1 i_2} + \rho_{st}^{\nu;k}(\theta).
$$
Taking derivatives of $\mu$ with respect to space in the strongly controlled relation for $\mu$ in \eqref{eq:mu_nu_str}, and using them in \eqref{eq:ca^1-ii} we have 
\begin{align}\label{eq:ca^1-iii}
\ca_{st}^{1;j}(\theta) &= \pt_{\eta^k} \mu_s^j(\nu_s(\theta)) F_{st}^{k} 
+ \dfrac{1}{2} \pt_{\eta^{k_1} \eta^{k_2}} \mu_s^j(\nu_s(\theta)) F_{st}^{k_1} F_{st}^{k_2}+ \lp \pt_{\eta^k} \mu_t^j(\nu_s(\theta)) - \pt_{\eta^k} \mu_s^j(\nu_s(\theta)) \rp F_{st}^{k}\nonumber \\
&+ \dfrac{1}{2} \lp \pt_{\eta^{k_1} \eta^{k_2}} \mu_t^j(\nu_s(\theta)) - \pt_{\eta^{k_1} \eta^{k_2}} \mu_s^j(\nu_s(\theta)) \rp F_{st}^{k_1} F_{st}^{k_2} + R_{st},
\end{align}
Collecting terms in \eqref{eq:ca^1-iii} and using assumption~\eqref{eq:zeta^2} we have that
\begin{align*}
\ca_{st}^{1;j}(\theta) &= \pt_{\eta^k} \mu_s^j(\nu_s(\theta)) \nu_s^{\bzeta; k i_1}(\theta) \der \bzeta_{st}^{i_1} \\
&+ \lp \pt_{\eta^k} \mu_s^j (\nu_s(\theta)) \nu_s^{\bzeta^2; k i_1 i_2}(\theta) + \pt_{\eta^{k_1} \eta^{k_2}}^2 \mu_s^j(\nu_s(\theta)) \nu_s^{\bzeta; k_1 i_1} \nu_s^{\bzeta; k_2 i_2} \right.\\
&\left.+ \quad\pt_{\eta^k} \mu_s^{\bzeta; j i_1}(\nu_s(\theta)) \nu_s^{\bzeta; k i_2}(\theta) + \pt_{\eta^k} \mu_s^{\bzeta; ji_2} (\nu_s(\theta))\nu_s^{\bzeta;k i_1}(\theta) \rp \bzeta_{st}^{2; i_1 i_2} + R_{st}.
\end{align*}
A similar result holds for $\ca_{st}^{2;j}(\theta)$ and combining these two relations we get our desired result.
\end{proof}

\section{Rough differential equation with drift.}\label{sec:rded} In this section we consider a rough path $\zeta$ satisfying Hypothesis~\ref{hyp:zeta}. Our aim is to consider a rough differential equation taking values in $\R^m$, of the form
\beq\label{eq:de}
dx_s = b(s, x_s) \, ds + \si(s, x_s) \, d\zeta_s, \quad s \in [0,T], \quad \text{ and }x_0= y,
\eeq
where the initial data $y$ is an element of $\R^m$ and $b:[0,T) \times \R^m \mapsto \R^m$, $\si:[0,T) \times \R^m \mapsto \R^{m,d}$. Notice that the solution to \eqref{eq:de} will be understood in the rough path sense. Namely a solution of \eqref{eq:de} is a continuous path $x \in \cac^{\al}([0,T]; \R^m)$ such that $x_0 = a$ and
\beq\label{eq:de-2}
\der x_{st} = \int_s^t b(u, x_u) \, du + \int_s^t \si(u, x_u) \, d\zeta_u, \qquad \text{ for all } s,t \in [0,T],
\eeq
and where the second integral is interpreted through the theory of integration of weakly controlled paths (see Section~\ref{sec:weak-contr}).
\begin{remark}\label{rmk:b-irregular}
Equations of the form \eqref{eq:de} are elaborations of more classical rough differential equations, allowing for a drift term $b$ and a time dependence in the coefficients $b$, $\si$. It should be observed that one could treat \eqref{eq:de} as a usual rough differential equation, by considering the time $t$ as a component of an enlarged rough path $(t, \zeta)$. However, in order to get better regularity conditions on the coefficients $b$, $\si$ it is convenient to consider the time component separately. This is what we do in Proposition~\ref{prop:exi+uni} below, whose proof is included for lack of a reference working under the same conditions 
(although one could argue that our case is covered by \cite{friz-hoquet-le} in a stochastic-rough context, with much longer proofs). Notice that those minimal regularity conditions will be needed when we consider equations involving a rough control.
Also observe that the decomposition \eqref{eq:b} for the solution to equation~\eqref{eq:de} will play an important role in our future considerations.
\end{remark}

\subsection{Existence and uniqueness result}
This section is devoted to state and prove our basic existence and uniqueness result with minimal continuity assumptions on the drift coefficient. We first label the main assumptions on our coefficients for further use.
\begin{hypothesis}\label{hyp:de-b,si}
In equation~\eqref{eq:de-2}, we suppose that the coefficients $b$ and $\si$ satisfy
$$
b \in \cac_b^{{0},1} ([0,T] \times \R^m; \R^{m}), \quad \text{and} \quad \si \in \cac_b^{1,2}([0,T] \times \R^m; \R^{m,d}).
$$ 
\end{hypothesis}
\begin{proposition}\label{prop:exi+uni}
Let $\al > \frac{1}{3}$ and suppose $\zeta \in \cac^{\al}([0,T]; \R^d)$ is a path satisfying Hypothesis~\ref{hyp:zeta}. Assume the coefficients $b$ and $\si$ satisfy Hypothesis~\ref{hyp:de-b,si}. Recall that the spaces $\cq_{\bzeta}^{\ka}$ are introduced in Definition~\ref{def:weakly-ctrld}. Then for every $\ka \in (\frac{1}{3}, \al)$, there exists an $x \in \cq_{\bzeta}^{\ka}([0,T]; \R^m)$ which solves \eqref{eq:de} in its integral form \eqref{eq:de-2}, where the RHS of \eqref{eq:de-2} has to be understood as in Proposition~\ref{prop:integral_as_weak}. If in addition $\si \in C_b^{1,3}([0,T]\times \R^m; \R^{m,d})$, the solution $x$ is also unique.
\end{proposition}

\begin{proof}
The proof is an elaboration of \cite[Proposition~7-8]{gubinelli} which takes into account the time components in $\si$ and $b$, as well as the drift term $b$. 
For sake of conciseness we will only focus on the uniqueness part. We divide the proof in several steps. 

\noindent\textbf{(i)} \emph{Setting}. 
We take two paths $x$ and $\tilde{x}$ in $\cq_{\bzeta}^{\ka}([0,T]; \R^m)$ solving the equation~\eqref{eq:de}, such that we also have $x_0 = \tilde{x}_0$. 
Let us write a solution $x$ as the root $x = G(x)$ for a map
$$
G: \cq_{\bzeta}^{\ka} ([0,T]; \R^m) \mapsto \cq_{\bzeta}^{\ka} ([0,T]; \R^m),
$$
where $G(x) = z$ is defined by
\beq\label{eq:z-G}
\der z_{st} = \int_s^t b(u, x_u) \, du + \int_s^t \si(u, x_u) \, d\zeta_u.
\eeq
Let us also specify the weakly controlled process decomposition for a process defined by \eqref{eq:z-G}. Namely assume that $x$ can be written as 
$$\der x = x^{\bzeta} \der \bzeta + r^{x}.$$
Next consider the path $w_u = \si(u, x_u)$. Proposition~\ref{prop:smooth_of_weak} implies $w \in \cq_{\zeta}^{\kappa}$ with the following decomposition:
\beq\label{eq:w-decomp}
\der w_{st} = w_s^{\zeta} \, \der \zeta_{st} + r_{st}^w,
\eeq
where $w^{\zeta}$ and $r^w$ are respectively defined by
\begin{eqnarray}\label{eq:w,r}
w_s^{\zeta} &=& \nabla_z \si(s, x_s) x^{\zeta} \nonumber \\
r_{st}^w &=& \nabla_z \si(s, x_s) r_{st}^x + [(\der\si(s,x))_{st} - \nabla_z \si(s, x_s) \der x_{st}] + {\der \si(\cdot, x_t)_{st}}.
\end{eqnarray}

Therefore gathering relations \eqref{eq:b} and \eqref{eq:w-decomp}, we get that the path $z$ in equation~\eqref{eq:z-G} satisfies
\beq\label{eq:z-2}
\der z_{st} = w_s \der \zeta_{st} + R_{st}^z = \si(s, x_s) \der \zeta_{st} + \nabla_z \si(s, x_s) \si(s, x_s) \mathbf{\zeta}_{st}^2 + \ci_{st}(\eta) + Q_{st}^z,
\eeq
where we have set
\beq\label{eq:eta-Q^z}
\eta_u = b(u, x_u), \quad \text{and} \quad Q^z = \laa(r^x \der \bzeta + \der x^{\bzeta} \bzeta^2).
\eeq
Similarly and with obvious notation, we get that the second solution $\tilde{x}$ to \eqref{eq:de} with 
$$\der \tilde{x} = \tilde{x}^{\bzeta} \der \bzeta + r^{\tilde{x}},$$ verifies the relation $\tilde{x} = G(\tilde{x})$ where $\tilde{z} = G(\tilde{x})$ can be decomposed as
\beq\label{eq:tilde-z-2}
\der \tilde{z}_{st} = \tilde{w}_s \der \zeta_{st} + R_{st}^{\tilde{z}} =\si(s, \tilde{x}_s) \der \zeta_{st} + \nabla_z \si(s, \tilde{x}_s) \si(s, \tilde{x}_s) \mathbf{\zeta}_{st}^2 + \ci_{st}(\tilde{\eta}) + Q_{st}^{\tilde{z}}.
\eeq
Let us introduce the notation 
\beq\label{eq:phi-norm}
\|\varphi\|_{\ast, T} := \cn[\varphi; \cq_{\zeta}^{\ka}([0,T]; V)]
\eeq
for an appropriate space $V$. 
We would like to show that
\beq\label{eq:obj}
\|z - \tilde{z}\|_{\ast, T} \leq c T^{\nu} \|x-\tilde{x}\|_{\ast, T}, 
\eeq
for some positive constants $c$ and $\nu$. In order to achieve \eqref{eq:obj}, we recall our definition~\eqref{eq:norm-controlled} for $|\cdot|_{\ast, T}$. Hence we will
obtain similar upper bounds separately for each of $\|w-\tilde{w}\|_{\ka, T}$, $\|w-\tilde{w}\|_{\infty, T}$, $\|z-\tilde{z}\|_{\ka, T}$ and $\|R^{z} - R^{\tilde{z}}\|_{2\ka, T}$ in the rhs of~\eqref{eq:z-2} and \eqref{eq:tilde-z-2}. We give a few details about those issues in the next step.

\vspace{0.1in}
\noindent
\textbf{(ii)} \emph{Bound for the derivatives}. Our first objective is to bound $\|w-\tilde{w}\|_{\ka, T}$. 
To this aim, recall that $w_s = \si(s, x_s)$, $\tilde{w}_s = \si(s, \tilde{x}_s)$. Since we wish to extract a factor $T^{\nu}$ from our estimate, we will use the controlled process structure of those two objects.
Namely invoking \eqref{eq:w-decomp} and~\eqref{eq:w,r} for $w$ and a similar decomposition for $\tilde{w}$ we get that
$$
\der (w - \tilde{w})_{st} = (w_s^{\zeta} - \tilde{w}_s^{\zeta}) \der \bzeta_{st} + (r_{st}^w - r_{st}^{\tilde{w}}),
$$
from which we easily deduce
\beq\label{eq:w-tildew-n1}
\|w-\tilde{w}\|_{\ka, T} 
\leq 
\| (w^{\bzeta} - \tilde{w}^{\bzeta}) \der \zeta\|_{\ka, T} + \|r^w - r^{\tilde{w}}\|_{\ka, T}. 
\eeq
In order to bound the rhs of \eqref{eq:w-tildew-n1}, let us write
\beq\label{eq:w-tildew-n2}
\| (w^{\bzeta} - \tilde{w}^{\bzeta}) \der \zeta\|_{\ka, T} \leq {\| (w^{\bzeta} - \tilde{w}^{\bzeta}) \|}_{\ka, T} {\|\der \zeta\|}_{\ka, T}. 
\eeq
Moreover, owing to the fact that $\bzeta \in \cac^{\al} ([0,T])$ with $\al > \ka$, we have ${\|\bzeta\|}_{\ka, T} \leq C_{\bzeta, T} T^{\al - \ka}$. Therefore one can write
\beq\label{eq:w-tildew-n3}
{\| (w^{\bzeta} - \tilde{w}^{\bzeta}) \der \bzeta \|}_{\ka, T} \leq C_{\bzeta, T} {\|w^{\bzeta} - \tilde{w}^{\bzeta}\|}_{\ka, T} T^{\al - \ka}.
\eeq

Next recalling that $w_u^{\bzeta} = \si(u, x_u)$, $\tilde{w}^{\bzeta} = \si(u, \tilde{x}_u)$ and invoking the $\cac_b^{1,3}$-regularity of $\si$, we easily get
\beq\label{eq:w-tildew-n4}
{\| w^{\bzeta} - \tilde{w}^{\bzeta} \|}_{\ka, T} \leq c_{\si} \|x - \tilde{x}\|_{\ka, T}.
\eeq
Gathering \eqref{eq:w-tildew-n3} and \eqref{eq:w-tildew-n4} into \eqref{eq:w-tildew-n2}, we have thus obtained
\beq\label{eq:w-tildew-n5}
\| (w^{\bzeta} - \tilde{w}^{\bzeta}) \der \bzeta \|_{\ka, T} \leq c_{\si, \bzeta, T} \|x - \tilde{x}\|_{\ka, T} T^{\al - \ka}.
\eeq 
For sake of conciseness, we leave the estimate of ${\|r - \tilde{r}\|}_{\ka, T}$ in \eqref{eq:w-tildew-n1} to the reader. It is a matter of standard rough paths analysis, based on tedious Taylor expansions. In the end we discover that
$$
\|r^w - r^{\tilde{w}}\|_{2\ka, T} \leq c_{\si, \bzeta, T} \|x - \tilde{x}\|_{\ast, T},
$$
where we recall that the notation $\|\cdot\|_{\ast, T}$ is defined by \eqref{eq:phi-norm}. Thus
\beq\label{eq:r-tilder-n1}
\|r^w - r^{\tilde{w}}\|_{\ka, T} \leq c_{\si, \bzeta, T} \| x- \tilde{x} \|_{\ast, T} T^{\ka}. 
\eeq
We now plug \eqref{eq:w-tildew-n5} and \eqref{eq:r-tilder-n1} into \eqref{eq:w-tildew-n1}, which yields 
\beq\label{eq:u-7}
\|w-\tilde{w}\|_{\ka, T} 
\leq 
c_{\si, \bzeta, T} \|x-\tilde{x}\|_{\ast, T} T^{\ka}.
\eeq
From \eqref{eq:u-7}, together with the fact that $w_0 = \tilde{w}_0$, we also obtain the following bound:
\beq\label{eq:u-8}
\|w-\tilde{w}\|_{\infty, T} \leq |w_0 - \tilde{w}_0| + T^{\ka} \|w-\tilde{w}\|_{\ka, T} \leq T^{2\ka} c_{\si, \bzeta, T} \|x-\tilde{x}\|_{\ast, T}.
\eeq

\vspace{0.1in}
\noindent
\textbf{(iii)} \emph{Global bound for $G(x)$}. 
To complete our bound for $\|z-\tilde{z}\|_{\ast, T}$, recall our decompositions \eqref{eq:z-2}-\eqref{eq:tilde-z-2}. In those expressions, we next obtain upper bounds for ${|R^z - R^{\tilde{z}}|}_{2 \ka, T}$ and ${\|z-\tilde{z}\|}_{\ka, T}$. Notice that 
$$
|\ci_{st}(\eta - \tilde{\eta})| \leq \int_s^t |b(u, x_u) - b(u, \tilde{x}_u)| du \leq C\int_s^t |x_u - \tilde{x}_u| du \leq C|t-s| \|x-\tilde{x}\|_{\infty, T}.
$$ 
Consequently we have that
\beq\label{eq:u-9}
\|\ci (\eta - \tilde{\eta})\|_{2\ka, T} \leq T^{1-2\ka} \|x-\tilde{x}\|_{\ast, T}.
\eeq
Let us now estimate the terms $Q^z$ defined by \eqref{eq:eta-Q^z}. Using Proposition~\ref{prop:La} and the fact that $\ka > \frac{1}{3}$, we have that
\begin{align} \label{eq:u-10}
\|Q^z - Q^{\tilde{z}}\|_{3\ka, T} &= {\| \laa((r^x - r^{\tilde{x}})\der \zeta + \der (x^{\zeta} - \tilde{x}^{\zeta}) \bzeta^2) \|}_{3\ka, T} \nonumber\\
&\leq c_{\ka} {\| (r^x -r^{\tilde{x}})\der \zeta + \der(x^{\zeta} - \tilde{x}^{\zeta})\bzeta^2 \|}_{3\ka, T}\nonumber\\
&\leq c_{\ka} \lp \|r^x - r^{\tilde{x}}\|_{2\ka, T} \|\zeta\|_{\ka} + \|x^{\zeta}-\tilde{x}^{\zeta}\|_{\ka} \|\bzeta^2\|_{2 \ka} \rp \nonumber\\
&\leq c_{\ka} C_{\zeta, T} \|x - \tilde{x}\|_{\ast, T}.
\end{align}
Now gathering \eqref{eq:u-9} and \eqref{eq:u-10} into \eqref{eq:z-2}-\eqref{eq:tilde-z-2} we have
\begin{align}\label{eq:u-11}
&\|R^z - R^{\tilde{z}}\|_{2\ka, T} = {\| (\nabla_z \si(\cdot, x) \si(\cdot, x) - \nabla_z \si(\cdot, \tilde{x}) \si(\cdot, \tilde{x})) \bzeta^2 + \ci(\eta - \tilde{\eta}) + (Q^{z} - Q^{\tilde{z}}) \|}_{2\ka, T} \nonumber\\
&\leq {\|\nabla_z \si(\cdot, x) \si(\cdot, x) - \nabla_z \si(\cdot, \tilde{x}) \si(\cdot, \tilde{x}) \|}_{\infty, T} \|\bzeta^2\|_{2 \ka} + \|\ci(\eta - \tilde{\eta})\|_{2\ka, T} + T^{\ka} \|Q^{z} - Q^{\tilde{z}}\|_{3\ka, T} \nonumber \\
&\leq T^{\ka} c_{\si} C_{\zeta, T} \|x-\tilde{x}\|_{\ast, T} + T^{1-2\ka} \|x-\tilde{x}\|_{\ast, T} + {T^{\ka}} C_{\ka, \bzeta, T} \|x-\tilde{x}\|_{\ast, T} \nonumber \\
 &\leq CT^{\ka} c_{\si} C_{\ka, \zeta, T} \|x-\tilde{x}\|_{\ast, T}.
\end{align} 
Furthermore owing to \eqref{eq:u-8} and \eqref{eq:u-11} we have
\begin{align}\label{eq:u-12}
\|z-\tilde{z}\|_{\ka, T} &\leq \|w-\tilde{w}\|_{\infty, T} \|\zeta\|_{\ka} + {\|R^z - R^{\tilde{z}}\|}_{\ka, T} \nonumber\\
&\leq T^{\ka} \|w-\tilde{w}\|_{\ka, T} \|\zeta\|_{\ka,T} + T^{\ka} \|R^{z} - R^{\tilde{z}}\|_{2\ka, T} \nonumber \\
&\leq 2T^{2\ka} c_{\si} C_{\ka, \zeta, T} \|x-\tilde{x}\|_{\ast, T}. 
\end{align}
Plugging \eqref{eq:u-7}, \eqref{eq:u-8}, \eqref{eq:u-11} and \eqref{eq:u-12} into definition \eqref{eq:norm-controlled}, we finally obtain our desired bound:
\begin{align*}
\|z-\tilde{z}\|_{\ast, T} &= \|w-\tilde{w}\|_{\infty, T} + \|w-\tilde{w}\|_{\ka, T} + \|R^z - R^{\tilde{z}}\|_{2\ka, T} + \|z-\tilde{z}\|_{\ka, T}\\
&\leq (3+T^{\ka})T^{\ka} c_{\si} C_{\zeta, T} \|x-\tilde{x}\|_{\ast, T} \leq 4T^{\ka} c_{\si} C_{\zeta, T} \|x-\tilde{x}\|_{\ast, T},
\end{align*}
for $T \leq 1$. Now choosing $T$ small enough we have
$$
\|z-\tilde{z}\|_{\ast, T} = \|G(x) - G(\tilde{x})\|_{\ast, T} \leq c \|x-\tilde{x}\|_{\ast, T}, 
$$
for some $0 < c < 1$. Thus for $T$ small enough, $G$ is a strict contraction in $\cq_{\zeta}^{\kappa} ([0,T]; \R^m)$ and thus has a unique fixed point. We can now obtain a global solution  belonging to $\cq_{\bzeta}^{\ka}$ on a general interval $[0,T]$ by patching together local solutions. This part of the proof is quite standard in rough paths analysis, and its details are left to the patient reader.
\end{proof}

\subsection{Application to the pathwise control setting}\label{sec:appli-pathwise-control}
Our ultimate aim is to consider a general pathwise control formulation which can be summarized as follows: we consider a rough differential equation \eqref{eq:rde-2} below, driven by a rough path $\bzeta$ and controlled by a path $\ga : [0, T] \to \cp(U)$. Here $\cp(U)$ is the set of Borel probability measures on $U \hb{\subseteq \R^k}$ and $(U,d)$ is a general control space which admits a metric $d$. Our differential equation can be written as
\beq\label{eq:rde-2}
d x_s^{\ga} = \int_U b(s, x_s^{\ga}, a)\ga_s(da)ds + \si(s, x_s^{\ga})d \zeta_s, \qquad x_0^{\ga} = y.
\eeq
In order to obtain existence and uniqueness of \eqref{eq:rde-2}, we will restrict $\ga$ to lie in a suitable subspace of measure-valued paths. 
Before specifying this hypothesis, let us establish the following notation.

\begin{notation}\label{not:P_2}
Let $\cp_2(U)$ be the set of probability measures on $U$ with finite second moments, equipped with the Wasserstein-2 metric
\beq\label{eq:W_2}
W_2(m_1, m_2) = \inf_{m \in \Gamma(m_1, m_2)}  {\lp \int \int d^2(x,y) dm(x,y)\rp}^{1/2},
\eeq
where $\Gamma(m_1, m_2)$ is the space of all couplings with marginals $m_1$ and $m_2$.
\end{notation}
Some of our arguments below will rely on compactness arguments. Therefore we will restrict our set of measures $\ga$ to be compact.

\begin{hypothesis}\label{hyp:ck}
In the sequel we shall consider a compact subset $\ck$ of the set $\cp_2(U)$ introduced in Notation~\ref{not:P_2}. 
\end{hypothesis}

{The set of measure-valued paths considered later will in fact enjoy some H\"older-continuity in time. This is to ensure that measure-valued paths driven by a rough path $\bzeta$ also fall under the purview of our analysis. Let us label another notation in order to describe this type of trajectories.}

\begin{notation}\label{not:ga}
Let $\ep, L$ be two strictly positive constants and consider the compact subset $\ck$ of Hypothesis~\ref{hyp:ck}. We define a set $\cv^{\ep,L}$ of measure-valued paths as follows:
$$
\cv^{\ep,L} = \lcl \ga:[0,T] \mapsto \ck: W_2(\ga_s, \ga_t) \leq L |t-s|^{\ep} \rcl. 
$$
In addition, we introduce a sup-distance between two paths $\ga^1, \ga^2 \in \cv^{\ep, L}$ in the following way:
\beq\label{eq:hh-a^2}
\hat{d}(\ga^1, \ga^2) = \sup_{t \in [0,T]} W_2(\ga_t^1, \ga_t^2),
\eeq
where the distance $W_2$ is introduced in \eqref{eq:W_2}.
\end{notation}
The class of measure-valued paths considered here is thus the following:
\begin{hypothesis}\label{hyp:ga}
Let $\ck$ be the compact set introduced in Hypothesis~\ref{hyp:ck}. We assume that there exists strictly positive parameters $\ep, L$  such that the path $\ga:[0,T] \mapsto \ck$ lies in the space $\cv^{\ep,L}$ introduced in Notation~\ref{not:ga}.
\end{hypothesis}

\begin{remark}
Note that if $(U,{d})$ is a compact metric space then $(\cp_2(U), W_2)$ is a compact metric space. Therefore when $U$ is compact we can just consider $\ck = \cp_2(U)$ as our state space for the path $\ga$. 
\end{remark}

As a consequence of Hypothesis~\ref{hyp:ck}, we obtain that the space $\cv^{\ep, L}$ of measure-valued H\"older paths is compact as well. We state and prove this technical result for sake of completeness.
\begin{lemma}
Let $\cv^{\ep, L}$ be the space introduced in Notation~\ref{not:ga},  and assume the space of measure-valued controls $\ck$ satisfies Hypothesis~\ref{hyp:ck}.
Then $\cv^{\ep, L}$, equipped with with the topology of uniform convergence, is a compact space.
\end{lemma}
\begin{proof}
In light of \cite[Theorem 47.1]{munkres}, we only need to show that $\cv^{\ep, L}$ is equicontinuous, point-wise relatively compact and closed. We will treat those properties separately.

\noindent
\emph{Equicontinuity.} Let $\ga \in \cv^{\ep, L}$, and fix $\delta>0$. Then for all $s, t$ such that $|s-t| < (\frac{\der}{L})^{1/\ep}$, we have
\begin{equation*}
W_2(\ga_s, \ga_t) \leq L |s-t|^{\ep} < \der \, .
\end{equation*}
This proves equicontinuity. 

\noindent
\emph{Point-wise relatively compact.} We need to show that for all $t \in [0,T]$, the subset of $\ck$ defined by $\cv_t^{\ep} = \{ \ga_t: \ga \in \cv^{\ep, L} \}$ is relatively compact. This is immediate since $\ck$ is compact.

\noindent
\emph{Closedness.} Let $\mu$ be a limit point of $\cv^{\ep, L}$. Then for every $n \in \N$ there exists $\mu^n \in \cv^{\ep, L}$ such that $\sup_{t \in [0,T]} W_2(\mu_t, \mu_t^n) < \frac{1}{n}$. Consequently, for every $n \in \N$ it holds that
$$
\dfrac{W_2(\mu_t, \mu_s)}{|s-t|^{\ep}} \leq \dfrac{W_2(\mu_t, \mu_t^n) + W_2(\mu_t^n, \mu_s^n) + W_2(\mu_s^n, \mu_s)}{|s-t|^{\ep}} \leq \dfrac{2}{n |s-t|^{\ep}} + L,
$$
where we use the fact that $\mu^n \in \cv^{\ep, L}$ implies $W_2(\mu_t^n, \mu_s^n) \leq L|s-t|^{\ep}$. Taking limit over $n$ yields $W_2(\mu_t, \mu_s) \leq L |s-t|^{\ep}$. We have thus proved that $\mu \in \cv^{\ep, L}$. 
\end{proof}

We now obtain an existence-uniqueness result for our equations involving measures. To this aim, let us label a new set of assumptions on our coefficients.
\begin{hypothesis}\label{hyp:rde-2-b,si}
In equation \eqref{eq:rde-2}, we assume that the coefficients $b$ and $\si$  are such that 
$$
b \in \cac_b^{0, 1,1}([0,T]\times \R^m \times U; \R^m), \quad \text{and} \quad \si \in \cac_b^{1, 3}([0,T] \times \R^m; \R^{m,d}).
$$
Furthermore $b(u,x,a)$ is uniformly continuous in $u$. 
\end{hypothesis}
Our existence-uniqueness result for coefficients depending on measures is obtained below as a corollary of Proposition~\ref{prop:exi+uni}.
\begin{corollary}\label{cor:eq-with-gamma}
Let $\al > \frac{1}{3}$ and suppose $\zeta \in \cac^{\al}([0,T]; \R^d)$ satisfying Hypothesis~\ref{hyp:zeta}. Assume $\ga$ satisfies Hypothesis~\ref{hyp:ga}, and that $b, \si$ fulfill Hypothesis~\ref{hyp:rde-2-b,si}. 
Then for every $\ka \in (\frac{1}{3}, \al)$ there exists an unique $x \in \cq_{\bzeta}^{\ka}([0,T]; \R^m)$ which solves \eqref{eq:rde-2}.
\end{corollary}
\begin{proof}
Define $\hat{b}(u, x) = \int b(u, x, a) \ga_u(da)$. 
Let us first check that $\hat{b}$ fulfills one of the conditions in Hypothesis~\ref{hyp:de-b,si}, namely $\hat{b} \in \cac_b^{0,1}$.
In order to show that $\hat{b}$ is continuous in $u$, we  use the following decomposition
\beq\label{eq:hatb-l}
\hat{b} (u_2, x) - \hat{b} (u_1, x) = \ca_{u_1 u_2}^1(x) + \ca_{u_1 u_2}^2(x),
\eeq
with
\begin{align}\label{eq:hatb-n}
\ca_{u_1 u_2}^1(x) &= \int_U b(u_1, x, a) [\ga_{u_2} - \ga_{u_1}](da), \nonumber \\
\ca_{u_1 u_2}^2(x) &= \int_U [b(u_2, x, a) - b(u_1, x, a)] \ga_{u_2}(da).
\end{align}
We now proceed to bound $\ca^1$ and $\ca^2$ above separately. For the term $\ca^1$ we write
$$
\ca_{u_1 u_2}^1(x) = \int_U b(u_1, x, a) \ga_{u_2}(da) - \int_U b(u_1, x, a) \ga_{u_1}(da).
$$
Furthermore we have assumed $a \mapsto b(u,x,a)$ to be a $\cac^1$-function uniformly in $(u,x) \in [0,T] \times \R^m$. Hence denoting by $\text{Lip}_1(U)$ the set of Lipschitz functions on $U$ with Lipschitz constant upper bounded by $1$, we have
$$
\lln \ca_{u_1 u_2}^1(x) \rrn \leq c_b \sup \lcl \int f(a) \ga_{u_2}(da) - \int f(a) \ga_{u_1}(da); f \in \text{Lip}_1(U) \rcl = c_b W_1(\ga_{u_1}, \ga_{u_2}),
$$
where we have resorted to the dual representation of the $1$-Wasserstein distance for the second identity. In addition, since $\ga_{u_1}$, $\ga_{u_2}$ are both probability measures, it is a well-known fact that the $W_1$-distance is dominated by the $W_2$-distance. Thanks to our Hypothesis~\ref{hyp:ga}, we thus obtain
\beq\label{eq:hatb-m}
\lln \ca_{u_1 u_2}^1(x) \rrn \leq c_b W_2(\ga_{u_1}, \ga_{u_2}) \leq c_{b,L} {|u_2 - u_1|}^{\ep} \to 0, \text{ as } u_2 \to u_1.
\eeq
The term $\ca_{u_1 u_2}^2(x)$ in \eqref{eq:hatb-n} is easily treated. Indeed, since $\ga_{u_2}$ is a probability measure, we have
\beq\label{eq:hatb-o}
\lln \ca_{u_1 u_2}^2(x) \rrn \leq \sup_{a \in U} \lln b(u_2, x, a) - b(u_1, x, a) \rrn \to 0, \text{ as } u_2 \to u_1,
\eeq
where the convergence is a consequence of Hypothesis~\ref{hyp:rde-2-b,si}. We now plug \eqref{eq:hatb-m} and \eqref{eq:hatb-o} into \eqref{eq:hatb-l}. We end up with the relation
$$
\lln \hat{b}(u_1, x) - \hat{b}(u_2, x) \rrn \to 0, \text{ as } u_2 \to u_1.
$$
Furthermore due to boundedness of $\frac{\partial b}{\partial x}$, $\hat{b}$ is differentiable in $x$. Thus we find that $\hat{b}$ satisfies the conditions in Proposition~\ref{prop:exi+uni}, and existence and uniqueness of \eqref{eq:rde-2} follows.
\end{proof}

We close this section by stating a series of continuity properties which will be an important technical step in the sequel.

\begin{proposition}\label{prop:cont-ga}
	We work under the conditions of Proposition~\ref{prop:exi+uni}. Consider 3 initial conditions $y, y^1,y^2 \in \mathbb R^m$ and 3 measure-valued controls $\gamma,\gamma^1,\gamma^2 \in \cv^{\ep,L}$. We denote by $x^{y_i}$ (resp. $x^{\gamma_i}$) the solution of eq.~\eqref{eq:rde-2} driven by the control $\gamma$ (resp. with initial condition $y$). Then, we have
	\beq\label{eq:hh-b}
	 \|x^{y_2} - x^{y_1}\|_{\infty} + \cn[x^{y_2}-x^{y_1}; Q_\zeta^\alpha([0,T];\mathbb R^m)] \leq C_{\zeta,\sigma,b,\gamma}(1+|y_1|+|y_2|) |y_2-y_1|,
	\eeq
	and 
	\beq\label{eq:hh-c}
	 \|x^{\gamma_2} - x^{\gamma_1}\|_{\infty} + \cn[x^{\gamma_2}-x^{\gamma_1}; Q_\zeta^\kappa([0,T];\mathbb R^m)] \leq (1+|y|) \hat C_{\zeta,\sigma,b,y} \hat d(\gamma^1,\gamma^2),
	\eeq
	where the distance $\hat d$ is defined in \eqref{eq:hh-a^2} above.
\end{proposition}

\begin{proof}
	A complete proof of proposition would be a long and tedious elaboration of Proposition~\ref{prop:exi+uni}. We will, therefore, omit those details for conciseness. Let us just briefly justify the less standard inequality (c) above. To this email, define $\hat b^i(t,x):= \int_U b(t,x,a) \ga_t^i(da)$ for $i=1,2$. It is readily checked that
$$
|\hat{b}^1(t,x) - \hat{b}^2(t,x)| \lesssim W_2 (\ga_t^1, \ga_t^2).
$$
Hence one can deduce (c) from \cite[Theorem 4.8]{friz-hoquet-le} specialized to a vanishing diffusion coefficient.
\end{proof}

\section{A HJB type equation}\label{sec:hjb}

In this section we derive a HJB equation for the value of an optimization problem in a rough environment. We start by introducing our problem setting and deriving some basic properties for the equations we are handling.

\subsection{Pathwise control setting}
In the pathwise optimal control context, one considers an optimiztion problem of the form 
\beq\label{eq:opt}
\sup \left\{J_T(\ga) \equiv \int_0^T  F(s, x_s^{\ga}, \ga_s) ds + {G(x_T^{\ga})};~\ga\in \cv^{\ep,L} \right\},
\eeq
where $F$ and $G$ satisfy Hypothesis~\ref{hyp:F,G} below, subject to a system like \eqref{eq:rde-2} controlled by a measure valued path $\ga$. 
\begin{hypothesis}\label{hyp:F,G}
$F:[0,T]\times \R^m \times \ck \mapsto \R$ is  bounded, Lipschitz continuous in $x$, uniformly over $\ga\in \ck$, and Lipschitz in $\ga \in \ck$ (where we recall from Notation~\ref{not:P_2} that $\ck \subset \cp_2(U)$ is equipped with the $W_2$-distance).   The function $G: \R^m \mapsto \R$ is bounded and continuous.
\end{hypothesis}

\begin{remark}\label{rem:entropy}
Suppose $U \subseteq \R^u$ for $u \geq 1$. Let $\ck \subset \cp_2(U)$ be a space of absolutely continuous measures with uniformly bounded second moments such that the density function $\dot{m}(u)$ for every $m \in \ck$ satisfies
$$
\|\nabla \log \dot{m} (u)\| \leq c_1 \|u\| + c_2,
$$
for some $c_1 >0$, $c_2 \geq 0$. 
Denote by $h(m)$ the differential entropy of a measure $m \in \ck$:
$$
h(m) = -\int_U {m}(u) \log {m}(u) du.
$$
Then by \cite[Proposition~1]{polyanski-wu}, the function $h$ satisfies
\beq\label{eq:entropy<W2}
\lln h({m}_1) - h({m}_2) \rrn \lesssim W_2(m_1, m_2).
\eeq
Otherwise stated, the function $h$ above is compatible with Hypothesis~\ref{hyp:F,G}.
\end{remark}

\begin{remark}
Let us give an important example of function $F$ in an exploratory setting, as introduced in \cite{tang-zhang-zhou, wang-zari-zhou}. Namely consider parameters $(s, x, \ga) \in [0,T] \times \R^m \times \ck$ as in Hypothesis~\ref{hyp:F,G} and Remark~\ref{rem:entropy}. Also pick a function $r: \R^m \times \ck \mapsto \R$ which is bounded, continuous in $x$ uniformly over $U$, as well as Lipschitz in $U$. Then $F$ is defined by
\beq\label{eq:a-F}
F(s,x,\ga) = \lp \int_U r(x, u) \dot\ga_s(u) du - \la \int_U \dot\ga_s(u)\log \dot\ga_s(u) du \rp .
\eeq
It is a direct consequence of \eqref{eq:entropy<W2} that F satisfies Hypothesis~\ref{hyp:F,G}. 
\end{remark}



An essential step in our methodology is to prove that the optimization problem~\eqref{eq:opt} can be solved. The main proposition in this direction is the following.
 
\begin{proposition}\label{prop:sup}
Let $F,G$ be functions satisfying Hypothesis~\ref{hyp:F,G}. Assume that Hypothesis~\ref{hyp:rde-2-b,si} is fulfilled so that Corollary~\ref{cor:eq-with-gamma} holds. The space $\ck$ is assumed to verify Hypothesis~\ref{hyp:ck}. We call $V_0(y)$ the supremum in \eqref{eq:opt}, that is
\begin{equation}\label{eq:j-a}
V_0(y) = \sup \{ J_T(\ga) ; \ga \in \cv^{\ep,L} \},
\end{equation}
where the dependence on $y \in \mathbb R^m$ comes from the initial condition in~\eqref{eq:rde-2} and where we recall that $\cv^{\ep,L}$ is introduced in Notation~\ref{not:ga}. 
Then, the supremum in \eqref{eq:j-a} is attained by an optimal $\ga^{\ast} \in \cv^{\ep, L}$. 
\end{proposition}
\begin{proof} Let $x^\gamma$ be the solution to~\eqref{eq:rde-2}. 
We define the map 
$$
\Gamma(\ga) = \int_0^T  F(s, x_s^{\ga}, \ga_s) ds +{G(x_T^{\ga})}.
$$
We will first show continuity of $\Gamma$ and our result will follow from the compactness of $\cv^{\ep, L}$. In order to prove continuity of $\gga$, consider a sequence of measures $\ga^n \in \cv^{\ep, L}$ converging to $\ga \in \cv^{\ep,L}$. Then we have the following decomposition:
\begin{equation*}
\int_0^T F(s, x_s^{\ga^n},\ga_s^n)ds - \int_0^T  F(s, x_s^{\ga}, \ga_s)ds
=
 I_1^n + I_2^n \, ,
\end{equation*}
where we have set
\begin{eqnarray*}
I_1^n&=&
\int_0^T \lc F(s, x_s^{\ga^n}, \ga_s^n) -F(s, x_s^{\ga}, \ga_s^n) \rc ds\\
I_2^n&=&
\int_0^T \lc F(s, x_s^{\ga}, \ga_s^n)  - F(s, x_s^{\ga}, \ga_s) \rc ds \, .
\end{eqnarray*}
Moreover,
using Proposition~\ref{prop:cont-ga} and Hypothesis~\ref{hyp:F,G} we have
$$
\lim_{n\to\infty} \|I_1^n\| \lesssim 
\lim_{n\to\infty} \int_0^T\|F(s, x_s^{\ga^n}, \cdot) - F(s, x_s^\ga, \cdot)\|_{\infty}\, ds
= 0 \, .
$$
In addition, due to Lipschitz continuity of $F$ in $\ga$ we discover that
$$
\lim_{n\to\infty} \| I_2^n \| 
\lesssim \hat{d}(\ga^n, \ga) = 0,
$$ 
Finally, due to Proposition~\ref{prop:cont-ga} we obtain that $G(x_T^{\ga^n})$ converges to $G(x_T^{\ga})$ by continuity of $G$. Gathering our 3 estimates, this completes the proof.
\end{proof}

\subsection{Heuristic derivation of a HJB equation}\label{sec:dpp}
In this section we will derive the HJB equation for the value related to our optimization problem~\eqref{eq:opt}. {Compared to \cite{diehl}, in our setting we have chosen an explicit approach. That is, we derive the HJB equation thanks to a direct application of rough path techniques.} 

We begin by introducing some notation. First, we generalize \eqref{eq:j-a} and set 
\beq\label{eq:hjb-a}
V(s, y) := \sup \lcl J_{sT}(\ga, y); \, \ga \in \cv^{\ep,L} ([s,T]) \rcl,
\eeq
where
\beq\label{eq:hjb-e}
J_{sT} (\ga, y) := \int_s^T F(r, x_r^{\ga}, \ga_r) \,  dr + G(x_T^{\ga}),
\eeq
$x_s^\gamma = y$ and where the set $\cv^{\ep, L}([s,t])$ is defined by
\beq\label{eq:hjb-b}
\cv^{\ep, L} ([s,T]) := \lcl \cac^{\ep}([s,T]; \ck); W_2(\ga_r, \ga_t) \leq L |t-r|^{\ep} \text{ for all } r,t \in [s,T]  \rcl.
\eeq
We also label a notation for our dynamics $x$ for further use.
\begin{notation}\label{not:hjb-c}
Consider a time $s \in [0,T]$ and an initial condition $y \in \R^m$. We call $x^{s, y, \ga} = \{ x_t^{s, y, \ga} ; t \in [s,T]\}$ the solution of the following rough differential equation (written in its integral representation),
\beq\label{eq:hh-a}
 x_t^{s, y, \ga} = y + \int_s^t \int_U b(r, x_r^{s, y, \ga}, a) \ga_r(da)dr + \int_s^t \si(r, x_r^{s, y, \ga}) d \bzeta_r \, .
\eeq
\end{notation}
In this context, the dynamic programming principle can be stated as follows.

\begin{proposition}\label{prop:dpp}
Under the conditions of Corollary~\ref{cor:eq-with-gamma}, let $\{ x^{s, y, \ga}; s \in [0,T], y \in \R^m \}$ be the family introduced in Notation~\ref{not:hjb-c}. Also recall that the value $V$ is defined by \eqref{eq:hjb-a}-\eqref{eq:hjb-e}, where $F,G$ satisfy Hypothesis~\ref{hyp:F,G}. Then the following relation holds true for $0 \leq s_1 \leq s_2 \leq T$:
\beq\label{eq:hjb-d}
V(s_1,y) = \sup \lcl \int_{s_1}^{s_2} F(r, x_r^{s_1, y, \ga}, \ga_r) dr + V(s_2, x_{s_2}^{s_1, y, \ga}); \ga \in \cv^{\ep, L}([s_1, T]) \rcl.
\eeq
\end{proposition}

\begin{proof}
Let us call $\hat{V}$ the right hand side of \eqref{eq:hjb-d}. We shall proceed by producing upper and lower bounds on $\hat{V}$. 

\noindent
\emph{Step 1: Lower bound.} Consider a given $\ga \in \cv^{\ep, L}([s,T])$. Then resorting to our notation \eqref{eq:hjb-e} we have
\beq\label{eq:hjb-f}
V(s_1,y) \geq J_{s_1 T}(\ga, y).
\eeq
Now we invoke the flow property for the dynamics $x$, which reads 
$$
x_{s_2+v}^{s_1, y, \ga} = x_v^{s_2, x_{s_2}^{s_1, y, \ga}}.
$$
Plugging this information into the right hand side of \eqref{eq:hjb-f}, we end up with
$$
V(s_1,y) \geq \int_{s_1}^{s_2} F(r, x_r^{s_1,y,\ga},\ga_r) dr + J_{s_2 T} (\ga, x_{s_2}^{s_1, y, \ga}) 
$$
Taking sup over $\ga \in \cv^{\ep, L}([s_1, T])$, we obtain 
\beq\label{eq:hjb-g}
V(s_1, y) \geq \hat{V}.
\eeq

\noindent
\emph{Step 2: Upper bound.} Fix $\ep > 0$, and pick a $\ga^{\ep} \in \cv^{\ep, L}([s_1,T])$ achieving a near maximum in $V(s_1, y)$. That is,
$$
V(s_1, y) \leq J_{s_1 T}(\ga^{\ep}, y) + \ep.
$$
This exists since the const function is continuous in $\gamma$.~
Then specifying the expression for $J$ one can write
$$
V(s_1, y) \leq \int_{s_1}^{s_2} F(r, x_r^{s_1, y, \ga^{\ep}}, \ga_r^{\ep}) dr + J_{s_2 T} (\ga^{\ep}, x_{s_2}^{s_1, y, \ga^{\ep}}) + \ep.
$$
By definition of $V$, we trivially have $J_{s_2 T}(\ga^{\ep}, x_{s_2}^{s_1, y, \ga^{\ep}}) \leq V(s_2, x_{s_2}^{s_1, y, \ga^{\ep}})$. Hence we obtain
$$
V(s_1, y) \leq \int_{s_1}^{s_2} F(r, x_r^{s_1, y, \ga^{\ep}}, \ga_r^{\ep})dr + V(s_2, x_{s_2}^{s_1, y, \ga^{\ep}}) + \ep.
$$
Since $\hat{V}$ is defined as the right hand side of \eqref{eq:hjb-d}, we thus get
$$
V(s_1, y) \leq \hat{V} +\ep.
$$
Since $\ep>0$ is arbitrary, it can be made as small enough such that we achieve the desired relation
\beq\label{eq:hjb-h}
V(s_1, y) \leq \hat{V}.
\eeq
The proof of~\eqref{eq:hjb-d} is now concluded by gathering \eqref{eq:hjb-g} and \eqref{eq:hjb-h}.
\end{proof}

Let us now turn to the heuristic version of a HJB type equation, which is obtained by assuming that $V$ is smooth enough.

\begin{proposition}\label{prop:hjb}
We assume that the conditions of Corollary~\ref{cor:eq-with-gamma} and Hypothesis~\ref{hyp:F,G} hold true. Suppose that the value function $V$ defined by \eqref{eq:hjb-d} is an element of $\cac^1([0,T] \times \R^n)$. Then $V$ solves the following first order equation:
\beq\label{eq:hjb-i}
\partial_t v(t,y) + \sup_{\ga \in \ck} H(t, y, \ga, \nabla v(t,y)) + \nabla v(t,y) \cdot \si(t,y) d \bzeta_t = 0
\eeq
for $(t,y) \in [0,T] \times \R^n$, with final condition 
\beq\label{eq:hjb-ii}
v(T,y) = G(y),
\eeq
and the Hamiltonian $H$ is defined by 
\beq\label{eq:hjb-iii}
H(t, y, \ga, p) = p \cdot \int_U b(t, y, a) \ga (da) +  F(t, y, \ga).
\eeq
\end{proposition}

\begin{proof}
We divide this proof into an upper bound and a lower bound. 

\noindent
\emph{Step 1: Upper bound.} Consider a solution to equation~\eqref{eq:rde-2} corresponding to a constant path $\ga_s \equiv \ga \in \ck$. In order to avoid ungainly notation, we simply write $x$ for this solution. Pick $0 \leq s_1 \leq s_2 \leq T$. Then according to \eqref{eq:hjb-d} we have
$$
V(s_1, x_{s_1}) \geq \int_{s_1}^{s_2} F(t, x_t, \ga) dt + V(s_2, x_{s_2}),
$$
or equivalently
\beq\label{eq:hjb-iv}
V(s_2, x_{s_2}) - V(s_1, x_{s_1}) \leq - \int_{s_1}^{s_2} F(t, x_t, \ga)dt
\eeq
Next divide \eqref{eq:hjb-iv} by $s_2 - s_1$ and take the limit $s_2 \searrow s_1$. Due to the fact that we assume $V$ to be differentiable we get
\beq\label{eq:hjb-vi}
{\left. \dfrac{d}{ds} V(s, x_s) \rrn}_{s=s_1} \leq - F(s_1, x_{s_1}, \ga). 
\eeq
Invoking the dynamics~\eqref{eq:rde-2} and the total derivative, morally we have
\beq\label{eq:hjb-v}
\partial_t V(s_1, x_{s_1}) + \nabla V(s_1, x_{s_1})\cdot \lp \int_U b(s_1, x_{s_1}, a)\ga(da) + \si(s_1, x_{s_1}) d \bzeta_{s_1} \rp
\leq - F(s_1, x_{s_1}, \ga).
\eeq
With our definition \eqref{eq:hjb-iii} for the Hamiltonian $H$ in hand, equation~\eqref{eq:hjb-v} can be read as
$$
\partial_t V(s_1, x_{s_1}) + H(s_1, x_{s_1}, \ga, \nabla V(s_1, x_{s_1})) + \nabla V(s_1, x_{s_1}) \cdot \si(s_1, x_{s_1}) d \bzeta_{s_1} \leq 0,
$$
which holds for all $\ga \in \ck$. Taking supremum in $\ga$ we have
\beq\label{eq:hjb-vii}
\partial_t V(s_1, x_{s_1}) + \sup_{\ga \in \ck}H(s_1, x_{s_1}, \ga, \nabla V(s_1, x_{s_1})) + \nabla V(s_1, x_{s_1}) \cdot \si(s_1, x_{s_1}) d \bzeta_{s_1} \leq 0.
\eeq
Otherwise stated, we have proved the upper bound in \eqref{eq:hjb-i}. 

\noindent
\emph{Step 2: Lower bound.} We now consider an arbitrarily small $\ka > 0$. For $s_1 \in [0,T]$, if $s_2 > s_1$ is such that $s_2 - s_1$ is small enough there exists $\ga \equiv \ga^{\ka, s_2}$ in $\cv^{\ep, L} ([s_1, T])$ such that
$$
V(s_1, y) - \ka (s_2 - s_1) \leq \int_{s_1}^{s_2} F(t, x_t, \ga_t) dt + V(s_2, x_{s_2}).
$$
Dividing this inequality by $s_2 - s_1$, we get
$$
\dfrac{V(s_2, x_{s_2}) - V(s_1, y)}{s_2 - s_1} + \dfrac{1}{s_2 - s_1} \int_{s_1}^{s_2} F(t, x_t, \ga_t) dt \geq -\ka.
$$
Then resorting to the fact that $V$ is differentiable as in \eqref{eq:hjb-vi}, we have
\begin{multline*}
\dfrac{1}{s_2 - s_1} \int_{s_1}^{s_2} \lcl \partial_t V(t, x_{t}) + \nabla V(t, x_{t}) \cdot \lp \int_U b(t, x_{t}, a) \ga_t(da) + \si(t, x_{t})d\bzeta_{t} \rp \right.\\
 + F(t, x_t, \ga_t) \bigg \} dt \geq -\ka.
\end{multline*}
Using the continuity properties of $b$ and $F$, as in Corollary~\ref{cor:eq-with-gamma} and Hypothesis~\ref{hyp:F,G} respectively, one can thus take limits $s_2 \searrow s_1$. In addition, observing that $H(t, x_t, \ga_t, \nabla V(t, x_t))$ converges to $\sup_{\ga \in \ck}H(s_1, x_{s_1}, \ga, \nabla V(s_1, x_{s_1}))$ as $t \searrow s_1$ (equivalently since the supremum is attained $\ga^{\ka, s_2}_{s_1} \to \argsup_{\ga \in \ck}H(s_1, x_{s_1}, \ga, \nabla V(s_1, x_{s_1}))$ as $s_2 \searrow s_1$), we obtain
\beq\label{eq:hjb-viii}
\partial_t V(s_1, x_{s_1}) + \sup_{\ga \in \ck}H(s_1, x_{s_1}, \ga, \nabla V(s_1, x_{s_1})) + \nabla V(s_1, x_{s_1}) \cdot \si(s_1, x_{s_1}) d \bzeta_{s_1} \geq - \ka.
\eeq
Since $\ka$ has been chosen arbitrarily small, we have achieved the desired upper bound.

\noindent
\emph{Step 3: Conclusion.} Gathering \eqref{eq:hjb-vii} and \eqref{eq:hjb-viii}, we have proved the desired identity \eqref{eq:hjb-i}. This completes our proof.
\end{proof}

Similar to what happens in very simple deterministic situations, the main issue with Proposition~\ref{prop:hjb} is that the value function $V$ in \eqref{eq:hjb-d} cannot be assumed to sit in $\cac^1([0,T] \times \R^m)$ in general. Let us specify the kind of regularity one might expect for $V$. 

\begin{proposition}\label{prop:V-reg}
Assume the conditions of Corollary~\ref{cor:eq-with-gamma} and Hypothesis~\ref{hyp:F,G} are met. In addition, assume that the function $G$ in~\eqref{eq:opt} is Lipschitz. Let $V$ be the value function defined by \eqref{eq:hjb-d}. Then for $(s_1, y_1),  (s_2, y_2) \in [0,T] \times \R^m$ we have 
\beq\label{eq:V-reg}
\lln V(s_2, y_2) - V(s_1, y_1) \rrn \leq C_{\bzeta, T} \lcl |y_1 - y_2| + (1+ |y_1| + |y_2|) |s_1 - s_2|^{\al} \rcl \, ,
\eeq
where $\al$ is the parameter in $(1/3, 1)$ featuring in Hypothesis~\ref{hyp:zeta}.
\end{proposition}

\begin{proof}
Taking into account the conditions of Corollary~\ref{cor:eq-with-gamma} and the fact that $G$ is Lipschitz, the functions $\varphi \equiv F, G, b$ and $\sigma$ satisfy 
\begin{align}~\label{eq:lip-cont}
	\begin{cases}
		&|\varphi(t,x,u) - \varphi(t,x',u)| \leq L |x-x'|, \quad\text{for all } t \in [0,T],~x,x'\in \mathbb R~\text{and}~u \in \mathcal P_2(U)\\
		&|\varphi(t,0,u)| \leq L , \quad\text{for all } t \in [0,T]~\text{and}~u\in\mathcal P_2(U),
	\end{cases}
\end{align}
where $L > 0$.

Next, let $x_t^{s,y,\gamma}$ represent a controlled dynamic starting at time $s$ at level $y$, as defined in expression \eqref{eq:hh-a} above. Therefore, according to our inequality~\eqref{eq:hh-b} in Proposition~\ref{prop:cont-ga} plus an easy consequence of~\eqref{eq:obj}, 
it follows that there exists a path dependent constant $K > 0$ such that for all $s_1,s_2 \in [0,T], \, y_1,y_2 \in \mathbb R, \, t \in [\max\{s_1,s_2\},T]$ and $\gamma \in \mathcal V^\varepsilon([\min\{s_2,s_2\},T])$ we have
\begin{equation}\label{c1}
|x_t^{s_1,y_1,\gamma} - x_t^{s_2,y_2,\gamma}| \leq K \left( |y_1 - y_2| + (1+ \max\{y_1|, |y_2|\})|s_1-s_2|^\alpha \right) \, .
\end{equation}
In addition, for every $(s,y) \in [0,T] \times \mathbb R$, $t \in [s,T]$ and $\gamma \in \mathcal V^\varepsilon([s,T])$, we also have
\begin{equation}\label{c2}
|x_t^{s,y,\gamma}| \leq K (1 + |y|) \, .
\end{equation}
Using the definition \eqref{eq:hjb-e} of the cost function and~\eqref{eq:lip-cont}-\eqref{c1}-\eqref{c2}, it follows that
\begin{align}
	\begin{cases}
		|J_{s_1T}(\gamma,y) - J_{s_2T}(\gamma,y)| &\leq K \left( |y_1-y_2| + ( 1 + \max\{|y_1|,|y_2|\})|s_1-s_2|^\alpha \right) \, ,
		\\
		&\qquad ~\text{for all } s_1,s_2 \in [0,T),~y_1,y_2 \in \mathbb R, \gamma \in \mathcal V^\varepsilon([\min\{s_2,s_2\},T]).
	\end{cases}
\end{align}
Inequality~\eqref{eq:V-reg} now follows automatically from the above relation.
\end{proof}

\subsection{Viscosity solutions.}\label{sec:visc} In view of Proposition~\ref{prop:V-reg}, we will now define a notion of viscosity solution to equation~\eqref{eq:hjb-i}. To this aim we first introduce a natural set of test functions related to our equation of interest. This set is similar to the class of $\cac^1$ functions considered in the deterministic theory, albeit perturbed by a noisy term involving $\bzeta$.

\begin{definition}\label{def:j-d}
Let $\si$ be a coefficient such that Hypothesis~\ref{hyp:rde-2-b,si} is fulfilled. We consider a path $\bzeta \in \cac^{\al}([0,T]; \R^d)$ with $\al > \frac{1}{3}$ verifying Hypothesis~\ref{hyp:zeta}. We say that $\psi \in \cac^{\al, 2}( [0,T] \times \R^m ; \R)$ is an element of the set of test functions $\scrc_{\si}$ if there exists a drift term $\psi^t \in \cac_b^{0,1} ([0,T] \times \R^m; \R)$ such that $\psi$ satisfies the following 
linear rough PDE:
\beq\label{eq:j-c}
\der \psi_{s_1 s_2}( y) =  \int_{s_1}^{s_2} \psi_r^t (y) dr - \int_{s_1}^{s_2} \nabla \psi_r (y) \cdot \si^k(r, y) d \zeta_r^k, \quad \text{for all }y\in \R^m, 
\eeq 
for all $y \in \R^m$ and $0 \leq s_1 < s_2 \leq T$. 
\end{definition}
\begin{remark}
In Definition~\ref{def:j-d} we have introduced test functions through a simple linear PDE. It might be closer to the jets perspective (and useful to our computations below) to define test functions through their rough paths expansion. Namely if we start from \eqref{eq:j-c}, the controlled paths expansion~\eqref{eq:weakly-controlled} of $\psi$ is
\beq\label{eq:f-a}
\der \psi_{s_1 s_2}(y) = - \nabla \psi_{s_1} (y) \cdot \si^k (s_1, y) \der \zeta_{s_1 s_2}^k + \rho_{s_1 s_2}(y),
\eeq
where $\rho$ is a remainder in $\cac_2^{2 \al}$. In addition, differentiating \eqref{eq:f-a} we get the following controlled path expansion for $\nabla \psi$:
\beq\label{eq:der-nabla-psi}
\der \nabla \psi_{s_1 s_2} (y) = - \lc \nabla \psi_{s_1}(y) \cdot \nabla \si^l(s_1, y) + \nabla^2 \psi_{s_1}(y) \si^l(s_1, y) \rc \der \bzeta_{s_1 s_2}^l + \nabla \rho_{s_1 s_2}(y).
\eeq
Therefore applying Proposition~\ref{prop:integral_as_weak}, we easily get that the strongly controlled path decomposition of $\psi$ is given by
\begin{multline}\label{eq:j-c-}
\der \psi_{s_1 s_2} (y) = \psi_{s_1}^t(y) (s_2 - s_1) 
-\nabla \psi_{s_1}(y) \cdot \si^k (s_1, y)  \der \bzeta_{s_1 s_2}^k\\ 
+ \lp \nabla \psi_{s_1}(y) \cdot \nabla \si^l(s_1, y)   +  \nabla^{2} \psi_{s_1}(y) \si^l(s_1, y) \rp \cdot \si^k(s_1, y)  \bzeta_{s_1 s_2}^{2; lk} + R_{s_1 s_2}^{\psi}(y) ,
\end{multline}
where $R^{\psi} \in \cac_2^{3\ka}(\R)$. Note that we have used vector and matrix notations in \eqref{eq:j-c-}, where $\si^k$ refers to the $k^{th}$ column of the matrix function $\si$. Alternatively \eqref{eq:j-c-} may be expressed using index notations as follows:
\begin{multline}\label{eq:j-c-ii}
\der \psi_{s_1 s_2}(y) = \psi_{s_1}^t(y)(s_2 -s_1) - \pt_i \psi_{s_1}(y) \si^{ik}(s_1, y) \der \bzeta_{s_1 s_2}^k \\
+\lp \pt_i \psi_{s_1}(y) \pt_j \si^{il}(s_1, y) + \pt_{ij}^2 \psi_{s_1}(y) \si^{il}(s_1, y) \rp \si^{jk}(s_1,y) \bzeta_{s_1 s_2}^{2; lk} + R_{s_1 s_2}^{\psi}(y),
\end{multline}
where we have used the following convention for sake of compact notation:
\begin{eqnarray}
\nabla \psi_s(y) \cdot \nabla \si^l(s, y) \cdot \si^k (s,y) &=& \pt_i \psi_s(y) \pt_j \si^{il} (s,y) \si^{jk} (s,y) 
\label{eq:mat_to_ind-a}  \\
\nabla^2 \psi_s(y) \si^l(s,y) \cdot \si^k(s,y) &=& \pt_{ij}^2 \psi_s(y) \si^{il} \si^{jk}(s,y).
\label{eq:mat_to_ind-b}
\end{eqnarray}
 \end{remark}
 \begin{remark}
In case of a deterministic optimal control problem, where the state process satisfies an ordinary differential equation, test functions for the notion of viscosity solutions are in the class of $\cac^1([0,T] \times \R^n)$ functions. In equation~\eqref{eq:j-c}, the term $\int_{s_1}^{s_2} \psi_r^t(y) dr$ is reminiscent of that structure. Then the noisy term in the right hand side of \eqref{eq:j-c} mimics the noisy term in \eqref{eq:hjb-i}.  
\end{remark}

Although the notion of test function introduced in Definition~\ref{def:j-d} is a natural one, it is not immediate that the class $\scrc_{\si}$ is nonempty. We thus state a proposition in this direction.

\begin{proposition}
Under the same conditions on $\si$ and $\zeta$ as in Definition~\ref{def:j-d}, consider a function $\psi^t \in \cac_b^{0, 1} ([0,T] \times \R^m)$. Then there exists a unique solution $\psi$ to equation~\eqref{eq:j-c} in the class $\cq_{\zeta}^{\ka}(\hb{\cac^1(\R^m)})$, where we recall that $\cq_{\zeta}^{\ka}$ is introduced in Definition~\ref{def:weakly-ctrld} and Remark~\ref{rem:j-e}. For a fixed value of $y$, equation~\eqref{eq:j-c} has to be understood as in Proposition~\ref{prop:exi+uni}. 
\end{proposition}

\begin{proof}
The proof is in fact a consequence of later developments. Namely, we will see that there is a one-to-one correspondence between elements of $\mathcal{T}_{\si}$ (seen as solutions to~\eqref{eq:j-c}) and elements $\hat \psi$ which are simply of the form 
\[
	\delta \hat \psi_{s_1 s_2}(y) = \int_{s_1}^{s_2} \hat \psi_r^t(y) dr.
\]
This correspondence is established in Proposition~\ref{prop:flow-c} by composing any $\psi$ solution to~\eqref{eq:j-c} with the rough flow defined by~\eqref{eq:flow-a}. Since the equation above has obviously a unique solution, the same is true for~\eqref{eq:j-c}. This completes the proof.
\end{proof}

With the notion of test function we have just introduced, a corresponding notion of viscosity solution can now be defined as below.

\begin{definition}\label{def:j-f}
Let $\zeta$ be a path in $\cac^{\al}([0,T]; \R^m)$ fulfilling Hypothesis~\ref{hyp:zeta}. The coefficients $b$ and $\si$ are assumed to satisfy Hypothesis~\ref{hyp:rde-2-b,si} as in the previous propositions. Consider a path $v \in \cq_{\zeta}^{\ka}(\hb{\cac^0(\R^m)})$. Recall that the set $\ct_{\si}$ is introduced in Definition~\ref{def:j-d}. We say that $v$ is a rough viscosity supersolution (resp. subsolution) of equation ~\eqref{eq:hjb-i} if
$$
v_T(y) \geq (\text{resp.} \leq)~ G(y), 
$$
and for every element $\psi$ of $\scrc_{\si}$ such that $v - \psi$ admits a local minimum (resp. maximum) at $(s,y)$, the drift $\psi^t$ in \eqref{eq:j-c} satisfies
$$
\psi_{s}^t(y) \leq (\text{resp.} \geq)~ -\sup_{\ga \in \ck} H(s, y, \ga, \nabla \psi_s(y)),
$$ 
where we recall that $H$ is defined by \eqref{eq:hjb-iii}. 
~We say that $v$ is a viscosity solution of equation~\eqref{eq:hjb-i} if it is both a viscosity supersolution and subsolution of equation~\eqref{eq:hjb-i}. 
\end{definition}

We now turn to our main result in this section, namely we prove that the value function $V$ is a rough viscosity solution of our HJB equation.

\begin{theorem}~\label{thm:v-hjb}
Let $V$ be the value function introduced in \eqref{eq:hjb-a}-\eqref{eq:hjb-e}. Then under the same conditions as for Definition~\ref{def:j-f}, $V$ is a rough viscosity solution of equation~\eqref{eq:hjb-i}. 
\end{theorem}

\begin{proof}
We divide this proof in several steps.

\noindent
\emph{Step 1: An upper bound.} Let $\psi \in \scrc_{\si}$ such that $V - \psi$ admits a local minimum at $(s_1, y)$ where we recall that $\ct_{\si}$ is given in Definition~\ref{def:j-d}. Recalling Notation~\ref{not:hjb-c}, we start a dynamics $x = x^{s_1, y, \ga}$ at $(s_1, y)$, with a constant $\ga$. Then, due to the fact that $(s,y)$ is a local minimum for $V-\psi$, for $s_2 > s_1$ and $s_2$ sufficiently close to $s_1$ we have
$$
V(s_2, x_{s_2}) - \psi_{s_2}( x_{s_2}) \geq V(s_1, y) - \psi_{s_1}( y),
$$
which can be recast as 
\beq\label{eq:j-g}
V(s_2, x_{s_2}) - V(s_1, y) \geq \psi_{s_{2}}( x_{s_2}) - \psi_{s_{1}}( x_{s_1}).
\eeq
In addition, from \eqref{eq:hjb-d} one can deduce that 
\beq\label{eq:j-h}
V(s_2, x_{s_2}) - V(s_1, y) \leq -\int_{s_1}^{s_2} F(r, x_r, \ga) dr.
\eeq
Hence gathering \eqref{eq:j-g} and \eqref{eq:j-h} we end up with
\beq\label{eq:j-i}
\der \Phi_{s_1 s_2} \leq - \int_{s_1}^{s_2} F(r, x_r, \ga) dr, \quad \text{where} \quad \Phi_s = \psi_{s}(x_s). 
\eeq

\noindent
\emph{Step 2: Rough path expansion for $\Phi$.} Let us expand the path $\Phi$ defined by \eqref{eq:j-i}. That is one can write
\beq\label{eq:j-m}
\der \Phi_{s_1, s_2} = \Phi_{s_1, s_2}^1 + \Phi_{s_1, s_2}^2,
\eeq
where the increments $\Phi^1$, $\Phi^2$ are respectively defined by
\begin{equation}\label{eq:j-j}
\Phi_{s_1 s_2}^1 = \psi_{s_{2}}( x_{s_{2}}) - \psi_{s_{1}}( x_{s_{2}}) \, ,
\quad\text{and}\quad
\Phi_{s_1 s_2}^2 =\psi _{s_{1}}(x_{s_2}) - \psi_{s_{1}}( x_{s_1}) \, .
\end{equation}
Notice that $\Phi^1$ in \eqref{eq:j-j} is simply the increment $\der \psi_{s_1 s_2}(x_{s_{2}})$. It thus seems natural to use equation~\eqref{eq:j-c} in order to express this term. However, the cancellations in the stochastic terms coming from  $\Phi^1$ and $\Phi^2$ will be better observed through controlled processes expansions. We shall thus resort to the decomposition~\eqref{eq:j-c-}, which yields
\begin{multline}\label{eq:j-k}
\Phi_{s_1 s_2}^1 = \psi_{s_1}^t (x_{s_2})(s_2 - s_1) 
- \nabla \psi_{s_1}(x_{s_2}) \cdot \si^k(s_1, x_{s_2}) \der \zeta_{s_1 s_2}^k \\
+\lp \nabla \psi_{s_1}(x_{s_2}) \nabla \si^k (s_1, x_{s_2})+  \nabla^2 \psi_{s_1}(x_{s_2}) \cdot \si^k(s_1, x_{s_2}) \rp \cdot  \si^l (s_1, x_{s_2}) \zeta_{s_1 s_2}^{2;kl} + R_{s_1 s_2}^{\Phi^1},
\end{multline}
where we recall our conventions \eqref{eq:mat_to_ind-a} and \eqref{eq:mat_to_ind-b} for the matrix products in \eqref{eq:j-k} and where $R^{\Phi^1} \in \cac_2^{3 \ka}([0,T];\R)$. 
In order to handle the term $\Phi^2$, let us first apply a change of variable formula for the controlled process $x$ composed with the differentiable function $\psi_{s_1}$. This yields
\beq\label{eq:Phi^2}
\Phi_{s_1 s_2}^2 = \int_{s_1}^{s_2} \nabla \psi_{s_{1}}( x_r) dx_r.
\eeq
It then seems natural to interpret the right hand side of \eqref{eq:Phi^2} in the light of Proposition~\ref{prop:integral_controlled}. However, due to the form of our expression~\eqref{eq:j-k}, it is more convenient to handle expansions in \eqref{eq:Phi^2} starting from $x_{s_2}$. To this aim we will consider for the rough path $\zeta$ above, the reversed time process $\check{\bzeta} := \{ \check{\bzeta}_t = \bzeta_{T-t}, \text{ for }t \in [0,T]\}$. Expanding equation~\eqref{eq:rde-2} with respect to the rough path $\check{\bzeta}$ and observing that $\der \check{\bzeta}_{uv} = - \der \bzeta_{T-v, T-u}$, we let the reader check that a second order expansion of $x$ is given by
\begin{multline}\label{eq:c--}
\der x_{s_1 s_2} = \si^k(s_2, x_{s_2}) \der \zeta_{s_1 s_2}^{k} + \si^l(s_2, x_{s_2}) \nabla \si^k(s_2, x_{s_2}) \bzeta_{s_1 s_2}^{2;kl}\\
 + \lp \int_U b(s_2, x_{s_2}, a) \ga(da) \rp (s_2-s_1) +R_{s_1 s_2}^x,
\end{multline}
where we have used matrix product conventions similar to \eqref{eq:mat_to_ind-a}-\eqref{eq:mat_to_ind-b} for notational sake. In order to avoid possible ambiguities, let us now write \eqref{eq:c--} with explicit indices:
%
\begin{multline*}
\der x_{s_1 s_2}^i = \si^{ik}(s_2, x_{s_2}) \der \bzeta_{s_1 s_2}^k + \si^{il}(s_2, x_{s_2}) \pt_j \si^{ik} (s_2, x_{s_2}) \bzeta_{s_1 s_2}^{2;kl} \\
+ \lp \int_U b(s_2, x_{s_2}, a) \ga(da) \rp (s_2-s_1)+ R_{s_1 s_2}^x,
\end{multline*}
where the remainder $R^{x}$ sits in $\cac_2^{3\ka}$. 
We are now ready to expand the right hand side $\int_{s_1}^{s_2} \nabla \psi_{s_1}(x_r) dx_r$ of \eqref{eq:Phi^2} by means of Proposition~\ref{prop:integral_controlled}. That is gathering the weakly controlled decomposition \eqref{eq:der-nabla-psi} for $\nabla \psi_{s_1}(x_r)$, the strongly controlled expression \eqref{eq:c--} for $x$ and applying a slight variation of Proposition~\ref{prop:integral_controlled} (taking into account the drift term in \eqref{eq:c--}) we get
\begin{align}\label{eq:phi2-2}
&\Phi_{s_1 s_2}^2 = \int_{s_1}^{s_2} \nabla \psi_{s_{1}}( x_r) dx_r \\
&= \nabla \psi_{s_1}(x_{s_2}) \cdot \lp \int_U b(s_2, x_{s_2}, a) \ga (da) \rp (s_2 - s_1)  + \nabla \psi_{s_1}(x_{s_2}) \cdot \si^k(s_2, x_{s_2}) \der \zeta_{s_1 s_2}^k \nonumber\\
 &-  \lp \nabla \psi_{s_1}(x_{s_2}) \cdot  \nabla \si^k(s_2, x_{s_2}) \si^l (s_2, x_{s_2})   + \nabla^2 \psi_{s_1} (x_{s_2}) \si^k(s_2, x_{s_2}) \cdot \si^l(s_2, x_{s_2}) \rp \bzeta_{s_1 s_2}^{2;kl} + R_{s_1 s_2}^{\Phi^2}, \nonumber
\end{align}
where the remainder $R^{\Phi^2}$ is an element of $\cac^{3\ka}([0,T];\R)$. In relation \eqref{eq:phi2-2} above, we have used our conventions \eqref{eq:mat_to_ind-a} and \eqref{eq:mat_to_ind-b} or matrix products again. For sake of completeness we now provide a version with indices:
\beq\label{eq:f-b}
\begin{aligned}
&\Phi_{s_1 s_2}^2 = \int_{s_1}^{s_2} \pt_i \psi_{s_1}(x_r) dx_r^i  =\pt_i \psi_{s_1}(x_{s_2}) \int_U b^i(s_2, x_{s_2},a) \ga(da) (s_2 - s_1)\\ 
&+ \pt_i \psi_{s_1} (x_{s_2}) \si^{ik} (s_2, x_{s_2}) \der \bzeta_{s_1 s_2}^k - \lp \pt_i \psi_{s_1}(x_{s_2}) \pt_j \si^{ik} (s_2, x_{s_2}) \si^{jl}(s_2, x_{s_2}) \right.\\
&\left.+ \pt_{ij}^2 \psi_{s_1} (x_{s_2}) \si^{jk}(s_2, x_{s_2}) \si^{il} (s_2, x_{s_2}) \rp \bzeta_{s_1 s_2}^{2;kl} + R_{s_1 s_2}^{\Phi^2},
\end{aligned}
\eeq
where $R$ is a remainder term in $\cac_2^{\mu}$ for some $\mu \ge 1$.
Finally using \eqref{eq:j-k} and \eqref{eq:phi2-2}, and recalling \eqref{eq:j-m} we obtain
\beq\label{eq:der_phi-g}
\der \Phi_{s_1 s_2} = \lc \psi_{s_1}^t (x_{s_2}) + \nabla \psi_{s_1}(x_{s_2}) \cdot \lp \int_U b(s_2, x_{s_2}, a)\ga(da) \rp \rc (s_2 - s_1) + R_{s_1 s_2}^{\Phi},
\eeq
where the term $R^{\Phi}$ is defined by
\begin{align}\label{eq:R_phi_e}
R_{s_1 s_2}^{\Phi} &= \nabla \psi_{s_1}(x_{s_2}) \cdot \lc \si^k(s_2, x_{s_2}) - \si^k(s_1, x_{s_2}) \rc \der \zeta_{s_1 s_2}^k \nonumber \\
&+ \nabla \psi_{s_1} (x_{s_2}) \lc \nabla \si^k(s_1, x_{s_2}) \si^l(s_1, x_{s_2}) -   \nabla \si^k(s_2, x_{s_2}) \si^l(s_2, x_{s_2})\rc \bzeta_{s_1 s_2}^{2;kl} \nonumber \\
&+ \nabla^2 \psi_{s_1}(x_{s_2}) \lc \si^k(s_1, x_{s_2}) \si^l(s_1, x_{s_2}) - \si^k(s_2, x_{s_2}) \si^l(s_2, x_{s_2}) \rc \bzeta_{s_1 s_2}^{2;kl} \nonumber \\
& + R_{s_1 s_2}^{\Phi^1} + R_{s_1 s_2}^{\Phi^2}.
\end{align}
We now claim that $R^{\Phi}$ is a remainder in $\cac_2^{\mu}([0,T]; \R)$ for a given $\mu > 1$. Indeed we are working under Hypothesis~\ref{hyp:rde-2-b,si}, which means in particular that $\si$ is a function in $\cac_b^{1,2}([0,T]\times \R^m; \R^{m,d})$. Hence it is readily checked that 
$$
\lln \nabla \psi_{s_1}(x_{s_2}) \cdot \lc \si^k(s_2, x_{s_2}) - \si^k(s_1, x_{s_2}) \rc \der \zeta_{s_1 s_2}^k  \rrn 
\leq c_{\si} \|\bzeta\|_{\al} \|\si\|_{\cac^{1,2}} \|\nabla \psi\|_{\infty} |s_2 - s_1|^{1+\al}.
$$
The other terms in the right hand side of \eqref{eq:R_phi_e} can be treated in a similar way. We conclude that there exists $\mu > 1$ such that
\beq\label{eq:R_phi_f}
R^{\Phi} \in \cac_2^{\mu}([0,T]; \R).
\eeq

\noindent
\emph{Step 3: An upper bound continued.} Let us gather our decomposition \eqref{eq:der_phi-g} and the analytic estimate \eqref{eq:R_phi_f}. One discovers that 
\beq\label{eq:phi_sum}
\lim_{s_2 \searrow s_1} \dfrac{\der \Phi_{s_1 s_2}}{s_2 - s_1}  = \psi_{s_1}^t(x_{s_1}) + \nabla \psi_{s_1} (x_{s_1}) \cdot \lp \int_U b(s_1, x_{s_1}, a) \ga(da) \rp.
\eeq
Using this information, one can also take limits in \eqref{eq:j-i} in order to get
\beq
\psi_{s_1}^t(x_{s_1}) + \nabla \psi_{s_1} (x_{s_1}) \cdot \lp \int_U b(s_1, x_{s_1}, a) \ga(da) \rp \leq - F(s_1, x_{s_1, \ga}).
\eeq
With definition~\eqref{eq:hjb-iii} for the Hamiltonian $H$ in mind and recalling that $x_{s_1} = y$, we have obtained that 
$$
\psi_{s_1}^t(y) \leq - H(s_1, y, \ga, \nabla \psi_{s_1}(y)).
$$
Furthermore we have chosen an arbitrary control $\ga \in \ck$, where we recall that the compact set $\ck$ is defined in Hypothesis~\ref{hyp:ck}. This yields, for every $s_1 \in [0,T)$,
\beq\label{eq:psi_s_j}
\psi_{s_1}^t(y) \leq \inf_{\ga \in \ck} -H(s_1, y, \ga, \nabla \psi_{s_1}(y)).
\eeq

\noindent
\emph{Step 4: A lower bound.} The lower bound for $\psi^t$ follows the same steps as the upper bound~\eqref{eq:psi_s_j}. Namely let $\psi \in \ct_{\si}$ such that $V-\psi$ attains a local maximum at $(s_1,y)$. 
Hence for an arbitrary $\gamma \in \cv^{\ep,L}([s,T])$, the corresponding dynamics $x = x^{s,y,\gamma}$ from Notation~\ref{not:hjb-c} and $s_2$ close enough to $s_1$ we have
\beq\label{eq:hh-e}
	 V(s_1,y) - \psi_{s_1}(y) - V(s_2,x_{s_2}) + \psi_{s_2}(x_{s_2}) \geq 0.
\eeq
In addition, going back to relation~\eqref{eq:hjb-d}, for any $\beta > 0$ one can choose $\gamma$ in $\cv^{\ep,L}([s,T])$ so that
\beq\label{eq:hh-f}
	\qquad V(s_1,y) - V(s_2,x_{s_2}) \leq \int_{s_1}^{s_2} F(r,x_r,\gamma_r) dr + \beta(s_2-s_1)
\eeq
Putting \eqref{eq:hh-e} and \eqref{eq:hh-f} together, we get the existence of a $\gamma \in \cv^{\ep,L}([s,T])$ such that
\beq\label{eq:phi_ub_1}
	\psi_{s_2}(x_{2}) - \psi_{s_1}(x_1) + \int_{s_1}^{s_2} F(r,x_r,\gamma_r) dr + \beta (s_2-s_1) \geq 0.
\eeq

Hence, recalling that we have set $\Phi_s = \psi_s(x_s)$ in \eqref{eq:j-i} and invoking the expansion \eqref{eq:der_phi-g} for $\der \Phi$, we obtain that $\der \Phi_{s_1 s_2} = \psi_{s_2}(x_{s_2}) - \psi_{s_1}(y)$ satisfies relation \eqref{eq:phi_sum}. Hence one can divide \eqref{eq:phi_ub_1} by $s_2 - s_1$ and take limits as $s_2 \to s_1$. This yields
\beq\label{eq:phi_sum_lim-k}
\psi_{s_1}^t(y) + \beta \geq -H(s_1, y, \ga_{s_1}, \nabla \psi_{s_1}(y)) \geq \inf_{\ga \in \ck} -H(s, y, \ga, \nabla \psi_{s_1}(y)).
\eeq
Eventually we remark that $\beta$ is an arbitrary small positive constant in \eqref{eq:phi_sum_lim-k}. Letting $\beta \searrow 0$ we obtain 
\beq\label{eq:psi_s_l}
\psi_{s_1}^t(y) \geq - \sup_{\ga \in \ck} H(s, y, \ga, \nabla \psi_{s_1}(y)).
\eeq

\noindent
\emph{Step 5: Conclusion.} Putting together \eqref{eq:psi_s_j} and \eqref{eq:psi_s_l}, we have obtained that for all $\psi \in \ct_{\si}$ such that $V - \psi$ admits a local minimum (resp. maximum) at $(s_1, y)$ we have
$$
\psi_{s_1}^t(y) \leq (\text{resp. } \geq) - \sup_{\ga \in \ck} H(s_1, y, \ga, \nabla_{s_1}(y)).
$$
This corresponds to the notion of rough viscosity solution introduced in Definition~\ref{def:j-f}, and finishes our proof.
\end{proof}

\begin{remark}
	Let us highlight the fact that our structure $\mathcal T_{\sigma}$ for test functions in Definition~\ref{def:j-d} has been chosen precisely so that the controlled path decompositions of $\Phi^1$ and $\Phi^2$ (resp. in~\eqref{eq:j-k} and~\eqref{eq:phi2-2}) cancel out up to terms of order 1. This point is key to identifying viscosity solutions.
\end{remark}

\subsection{Rough Flows}\label{sec:flows}

The uniqueness of solution to equation \eqref{eq:hjb-i} will be obtained by formulating an equivalent equation where we get rid of the rough term. We prepare the ground for this transformation in the current section.

\begin{proposition}\label{prop:flow-a}
Let $\bzeta$ be a path satisfying Hypothesis~\ref{hyp:zeta}. For $\eta \in \R^m$ and $\si \in C_b^{1,3}([0,T]\times \R^m; \R^{m,d})$, consider the unique solution $\phi$ to the following equation:
\beq\label{eq:flow-a}
\phi_t(\eta) = \eta + \int_0^t \si^k (r, \phi_r(\eta)) d \bzeta_r^k.
\eeq
Then 
$\phi(\eta)$ is well-defined as a strongly controlled process in $\tilde{\cq}_{\bzeta}^{\ka}$ with decomposition
\beq\label{eq:phi-st-dec}
\phi_s^{\bzeta; k}(\eta) = \si^k(s, \phi_s(\eta)), \quad
\phi_s^{\bzeta^2; kl}(\eta) = \nabla \si^k(s, \phi_s(\eta)) \si^l(s, \phi_s(\eta)).
\eeq
or using indices
$$
\phi_s^{\bzeta; j i_1}(\eta) = \si^{j i_1}(s, \phi_s(\eta)), \quad
\phi_s^{\bzeta^2; j i_1 i_2}(\eta) = \pt_{\eta^k} \si^{j i_1}(s, \phi_s(\eta)) \si^{k i_2}(s, \phi_s(\eta)).
$$
\end{proposition}

\begin{proof}
Notice that the uniqueness of solution to \eqref{eq:flow-a} is a consequence of Proposition~\ref{prop:exi+uni}.
It is readily checked from \eqref{eq:flow-a} that we have the following decomposition for $\phi(\eta)$, namely
$$
\der \phi_{s_1 s_2}(\eta) = \si^k(s_1, \phi_{s_1}(\eta)) \der \zeta_{s_1 s_2}^{k} + \si^l(s_1, \phi_{s_1}(\eta)) \nabla \si^k(s_1, \phi_{s_1}(\eta)) \bzeta_{s_1 s_2}^{2;kl}\\
  +R_{s_1 s_2}^{\phi},
$$
where the remainder $R^{\phi}$ sits in $\cac_2^{3\ka}$. The proof now follows from Definition~\ref{def:str-contr}. 
\end{proof}

We now derive an equation for the inverse of $\phi$.
\begin{proposition}\label{prop:flow-b}
Let $\bzeta$ be a path satisfying Hypothesis~\ref{hyp:zeta}. For $\eta \in \R^m$ and $\si \in C_b^{1,3}([0,T]\times \R^m; \R^{m,d})$, consider the unique solution to the following equation:
\beq\label{eq:flow-b}
\chi_t(\eta) = \eta - \int_0^t \nabla \chi_r(\eta) \cdot \si^k(r, \eta) d \bzeta_r^k.
\eeq
Then
\begin{enumerate}[wide, labelwidth=!, labelindent=0pt, label=\textnormal{(\arabic*)}]
\setlength\itemsep{.02in}
\item\label{it:flow-b-i} $\chi(\eta)$ is well-defined as a strongly controlled process in $\tilde{\cq}_{\ka}^{\bzeta}(\R^m)$ with decomposition:
\begin{align}\label{eq:chi-st-dec}
\chi_s^{\bzeta; k}(\eta) &= - \nabla \chi_s(\eta) \cdot \si^k(s, \eta),\\
\chi_s^{\bzeta^2;kl} (\eta) &= \lp \nabla \chi_s(\eta) \cdot \nabla \si^l(s, \eta) + \nabla^2 \chi_s(\eta) \si^l(s, \eta) \rp \cdot \si^k(s, \eta).
\end{align}
or using indices
\begin{align*}
\chi_s^{\bzeta; j i_1}(\eta) &= - \pt_{\eta^k} \chi_s^j (\eta) \si^{k i_1}(s, \eta),\\
\chi_s^{\bzeta^2;j i_1 i_2} (\eta) &= \lp \pt_{\eta^{k_1}} \chi_s^j(\eta) \pt_{\eta^{k_2}} \si^{k_1 i_2}(s, \eta) + \pt_{\eta^{k_1} \eta^{k_2}}^2 \chi_s^j(\eta) \si^{k_1 i_2}(s, \eta) \rp  \si^{k_1 i_1}(s, \eta).
\end{align*}
\item\label{it:flow-b-ii} For all $t \in [0,T]$ we have $\chi_t = \phi_t^{-1}$, where $\phi$ is the solution to \eqref{eq:flow-a}. 
\end{enumerate}
\end{proposition}

\begin{proof}
For Item~\ref{it:flow-b-i}, the forms of the coefficients in the strongly controlled representation of $\chi(\eta)$ is readily observed from that of $\psi$ in \eqref{eq:j-c} in conjunction with equation \eqref{eq:j-c-}, when the drift term $\psi^t$ is taken to be zero. We now proceed to show item~\ref{it:flow-b-ii} through an application of Proposition~\ref{prop:str_comp}. For notational sake, we show the details for the case when $m=d=1$ and suppress the dependance on $\eta$ for notational convenience. 
Indeed, using Proposition~\ref{prop:str_comp} we observe that
\beq\label{eq:f-c}
{\lp \chi \circ \phi \rp}_s^{\bzeta} = \chi_s^{\bzeta} (\phi_s) + \pt_{\eta} \chi_s(\phi_s) \phi_s^{\bzeta} .
\eeq
Replacing the values of $\phi^{\bzeta}$ and $\chi^{\bzeta}$ from \eqref{eq:phi-st-dec} and \eqref{eq:chi-st-dec} respectively, we obtain
\beq\label{eq:f-d}
{\lp \chi \circ \phi \rp}_s^{\bzeta}
= - \pt_{\eta} \chi_s(\phi_s) \si (s, \phi_s) + \pt_{\eta} \chi_s(\phi_s) \si(s, \phi_s) = 0.
\eeq
Furthermore, again from Proposition~\ref{prop:str_comp} we have
\beq\label{eq:f-e}
{\lp \chi \circ \phi \rp}_s^{\bzeta^2} 
=  \chi_s^{\bzeta^2} (\phi_s) + \pt_{\eta} \chi_s(\phi_s) \phi_s^{\bzeta^2}  + 2 \pt_{\eta} \chi_s^{\bzeta} (\phi_s) \phi_s^{\bzeta} + \pt_{\eta^2}^2 \chi_s(\phi_s) (\phi_s^{\bzeta})^2.
\eeq
Replacing the values of $\chi^{\bzeta}$, $\chi^{\bzeta^2}$, $\phi^{\bzeta}$ and $\phi^{\bzeta^2}$ from \eqref{eq:phi-st-dec} and \eqref{eq:chi-st-dec} we obtain
\begin{align}\label{eq:f-f}
&{\lp \chi \circ \phi \rp}_s^{\bzeta^2}  = \lc \pt_{\eta} \chi_s(\phi_s) \pt_{\eta} \si(s, \phi_s) + \pt_{\eta^2}^2 \chi_s(\phi_s) \si(s, \phi_s) \rc \si(s, \phi_s)+\pt_{\eta} \chi_s(\phi_s) \pt_{\eta} \si(s, \phi_s) \si(s, \phi_s) \nonumber\\
 &+ 2\pt_{\eta} \lc -\pt_{\eta} \chi_s(\phi_s) \si(s, \phi_s) \rc \si(s, \phi_s)  +\pt_{\eta^2}^2\chi_s(\phi_s) (\si(s, \phi_s))^2=0.
\end{align}
Thus $\chi \circ \phi$ must be the constant process. Since $\chi_0(\phi_0(\eta)) = \eta$, we must have
$\chi_t(\phi_t(\eta)) = \eta$, for all $t \in [0,T]$ and $\eta \in \R$. This proves $\chi_t(\eta) = \phi_t^{-1}(\eta)$. 
\end{proof}

\begin{proposition}\label{prop:flow-diff}
Let $\phi$ be the solution to \eqref{eq:flow-a} as defined in Proposition~\ref{prop:flow-a}. Then for any $t \geq 0$, $\phi_t$ is a diffeomorphism.
\end{proposition}

\begin{proof}
Note that according to \cite[Ch. 11]{friz-victoir} under the assumptions on $\si$ in \eqref{eq:flow-a}, $\phi_t$ is a differentiable function.
This implies that $k_t(\eta) := \nabla \phi_t(\eta)$ satisfies the following equation:
\beq\label{eq:k_t}
k_t(\eta) = I_d + \int_0^t \nabla \si^k(r, \phi_r(\eta)) k_r (\eta) d \bzeta_r^k,
\eeq
where $I_d$ is the identity matrix in $\R^{d,d}$. We claim that $a_t = k_t^{-1}$ satisfies the following equation:
\beq\label{eq:a_t}
{a_t (\eta)} = I_d - \int_0^t a_r (\eta) \nabla \si^k(r, \phi_r(\eta)) d \bzeta_r^k.
\eeq
The proof of this claim is a variant of \cite[Proposition~4.11]{friz-victoir} taking into account the time dependence in the coefficients. For simplicity and in order to avoid lengthy expressions, we will prove \eqref{eq:a_t} for $m=d=1$, and let the patient reader work out the details for the general situation. Indeed, \eqref{eq:k_t} and \eqref{eq:a_t} are now respectively
\begin{align}
k_t (\eta ) &= 1 + \int_0^t \pt_{\eta} \si(r, \phi_r(\eta)) k_r (\eta) d \zeta_r, \label{eq:k_t-1d} \\
a_t  (\eta) &= 1 - \int_0^t  a_r  (\eta)  \pt_{\eta} \si(r, \phi_r(\eta)) d \zeta_r. \label{eq:a_t-1d}
\end{align}
Along the same lines as for \eqref{eq:z-2}, it is readily checked that both $k_t$ and $a_t$ are strongly controlled processes in $\tilde{\cq}_{\bzeta}^{\ka}$ with decomposition:
\begin{align}\label{eq:k_t-str-dec}
k_s^{\bzeta} (\eta) &= \pt_{\eta} \si(s, \phi_s(\eta)) k_s(\eta), \nonumber \\
k_s^{\bzeta^2} (\eta) &= \lp \pt_{\eta^2}^2 \si(s, \phi_s(\eta)) \si(s, \phi_s(\eta)) + (\pt_{\eta} \si(s, \phi_s(\eta)))^2 \rp k_s(\eta),
\end{align}
and
\begin{align}\label{eq:a_t-str-dec}
a_s^{\zeta}(\eta) &= -a_s(\eta) \pt_{\eta} \si(s, \phi_s(\eta)), \nonumber \\
a_s^{\bzeta^2}(\eta) &= \lp (\pt_{\eta} \si(s, \phi_s(\eta)))^2 - \pt_{\eta^2}^2 \si(s, \phi_s(\eta)) \si(s, \phi_s(\eta)) \rp a_s(\eta).
\end{align}
Let $L_t(\eta) = k_t(\eta) a_t(\eta)$. Suppressing the dependence on $\eta$ for notational convenience and resorting to \eqref{eq:k_t-str-dec}-\eqref{eq:a_t-str-dec}, the strongly controlled process decomposition for $L$ can be found as follows:
\begin{align*}
\der L_{st} &= \der k_{st}  a_t  + k_s  \der a_{st} \\
&= \lp k_s^{\zeta} \der \bzeta_{st} + k_s^{\zeta^2} \bzeta_{st}^2 + \rho_{st}^k \rp \lp a_s + a_s^{\bzeta} \der \bzeta_{st} + a_s^{\bzeta^2}  \bzeta_{st}^2 + \rho_{st}^a \rp + k_s \lp a_s^{\bzeta} \der \bzeta_{st} + a_s^{\bzeta^2} \bzeta_{st}^2 + \rho_{st}^a \rp \\
&= L_s^{\bzeta} \der \bzeta_{st} +  L_{st}^{\bzeta^2} \bzeta_{st}^2 + \rho_{st}^L,
\end{align*}
where $\rho^L \in \cac^{\mu}$ for some $\mu > 1$, and
\beq\label{eq:L-coeff}
L_s^{\bzeta} = \lp k_s^{\zeta} a_s + k_s a_s^{\zeta} \rp, \qquad
L_{s}^{\bzeta^2} = \lp k_s^{\bzeta^2} a_s + k_s a_s^{\bzeta^2} + 2 k_s^{\bzeta} a_s^{\bzeta} \rp.
\eeq
Plugging the values of the coefficients in \eqref{eq:k_t-str-dec}-\eqref{eq:a_t-str-dec} in \eqref{eq:L-coeff} we find that
$$
L_s^{\bzeta} = L_s^{\bzeta^2} = 0. 
$$
Thus $L$ must be a constant process. Since $k_0(\eta) = a_0(\eta) = 1$ we must have that $L_t(\eta) = k_t(\eta) a_t(\eta) = 1$ for all $t \in [0,T]$ and $\eta \in \R$,
 that is, $a_t(\eta)$ is the inverse of $k_t(\eta) = \nabla \phi_t(\eta)$. 
 
Summarizing our considerations so far, we have found that for every $t \geq 0$, the Jacobian $k_t = \nabla \phi_t$ is invertible.
Combined with Proposition~\ref{prop:flow-b}~Item\ref{it:flow-b-ii}, where we have showed that $\phi_t$ is invertible, we conclude that
 $\phi_t$ is a diffeomorphism for all $t \geq 0$. This finishes our proof.
\end{proof}

\begin{remark}
In the proofs of Propositions~\ref{prop:flow-b} and \ref{prop:flow-diff} as well as in the following proof in this section, we shall perform most of our rough paths computations through expansions of order~2. It might sometimes (as in Proposition~\ref{prop:flow-diff}) be more elegant to invoke rough differential equations and change of variables formulae. We have sticked to expansions here, since they are more compatible with the definition of jets. Those are natural notions in viscosity theory for PDEs, and will be developed for \eqref{eq:hjb-i} in subsequent contributions.
\end{remark}

\subsection{Uniqueness of the viscosity solution}\label{sec:visc-uniq}
The rough flow considerations developed in the previous section enable us to prove uniqueness for equation~\eqref{eq:hjb-i}, by transferring the initial rough problem to a deterministic problem with rough coefficients. We begin this procedure with a map for test functions.

\begin{proposition}\label{prop:flow-c}
Consider the vector fields $\si$ as in Propositions~\ref{prop:flow-a}-\ref{prop:flow-b}, and the set $\ct_{\si}$ of test functions as in Definition~\ref{def:j-d}. 
\begin{enumerate}[wide, labelwidth=!, labelindent=0pt, label=\textnormal{(\arabic*)}]
\setlength\itemsep{.02in}

\item\label{it:A} Consider the solution $\phi$ to equation~\eqref{eq:flow-a}. Then for any $\psi \in \ct_{\si}$, the composition $\psi \circ \phi$ is an element of $\ct_0$. 

\item\label{it:B} Consider the solution $\phi^{-1}$ to equation \eqref{eq:flow-b}. Then for any $\psi \in \ct_0$, the composition $\psi \circ \phi^{-1}$ is an element of $\ct_{\si}$. 

\end{enumerate}
\end{proposition}

\begin{proof}
We will only prove Item~\ref{it:A}, since the proof of Item~\ref{it:B} is very similar. Observe that for $\eta \in \R^m$, $\psi(\eta)$ and $\phi(\eta)$ are both strongly controlled processes. Consequently, Proposition~\ref{prop:str_comp} implies that $\psi \circ \phi $ is a strongly controlled process in $\tilde{\cq}_{\ka}^{\bzeta}(\R^m)$. In addition, we claim that the coefficients for $(\psi \circ \phi)(\eta)$ vanish when $\psi \in \ct_{\si}$. We prove this claim for the special case when $m=d=1$ and leave the details of the general case to the patient reader. 
Note that from \eqref{eq:j-c-} (and along the same lines as for \eqref{eq:j-k}) we have the following coefficients for any $\psi \in \ct_{\si}$:
\beq\label{eq:psi-st-dec}
\psi_s^{\bzeta}(\eta) = - \pt_{\eta} \psi_s(\eta) \si(s, \eta), \quad \psi_s^{\bzeta^2} (\eta) = \lp \pt_{\eta} \psi_s(\eta) \pt_{\eta} \si(s, \eta) + \pt_{\eta^2}^2 \psi_s(\eta) \si(s, \eta) \rp \si(s, \eta). 
\eeq
We now proceed similarly to what we did in Proposition~\ref{prop:flow-b}. Namely we invoke decomposition~\eqref{eq:phi-st-dec} for $\phi$, relation~\eqref{eq:psi-st-dec} for $\psi$ and Proposition~\ref{prop:str_comp} for the composition $\psi \circ \phi$. Along the same lines as \eqref{eq:f-c} and \eqref{eq:f-d} we get
\beq\label{eq:psi(phi)-1}
(\psi \circ \phi)_s^{\bzeta} = \psi_s^{\bzeta}(\phi_s) + \pt_\eta \psi_s(\phi_s) \phi_s^{\bzeta}= - \pt_\eta \psi_s(\phi_s) \si (s, \phi_s) + \pt_\eta \psi_s(\phi_s) \si(s, \phi_s) =0.
\eeq
In addition, following the same steps as for \eqref{eq:f-e}-\eqref{eq:f-f} we obtain
\begin{align}\label{eq:psi(phi)-2}
(\psi \circ \phi)_s^{\bzeta^2} &= \psi_s^{\bzeta^2} (\phi_s) + \pt_{\eta} \psi_s(\phi_s) \phi_s^{\bzeta^2} + 2 \pt_{\eta} \psi_s^{\bzeta} (\phi_s) \phi_s^{\bzeta} + \pt_{\eta^2}^2 \psi_s(\phi_s) (\phi_s^{\bzeta})^2 \nonumber \\
&= \lp \pt_{\eta} \psi_s(\phi_s) \pt_{\eta} \si(s, \phi_s) + \pt_{\eta^2}^2 \psi_s(\phi_s) \si(s, \phi_s) \rp \si(s, \phi_s) + \pt_{\eta} \psi_s(\phi_s) \pt_{\eta} \si(s, \phi_s) \si(s, \phi_s) \nonumber\\
&+ 2 \pt_{\eta} \lp -\pt_{\eta} \psi_s(\phi_s) \si(s, \phi_s) \rp \si(s, \phi_s) + \pt_{\eta^2}^2 \psi_s(\phi_s) (\si(s, \phi_s))^2=0.
\end{align}
Comparing with \eqref{eq:psi-st-dec} we have from \eqref{eq:psi(phi)-1}-\eqref{eq:psi(phi)-2} that $\psi \circ \phi$ is an element of $\ct_0$. 
This proves Item~\ref{it:A}. 
\end{proof}
Thanks to Proposition~\ref{prop:flow-c}, we can now map the solution to the rough viscosity equation~\eqref{eq:hjb-i} to the solution of a PDE with random coefficients.
\begin{proposition}\label{prop:flow-det}
As in Proposition~\ref{prop:flow-b}, consider a  path $\bzeta$ satisfying Hypothesis~\ref{hyp:zeta} and a coefficient $\si \in C_b^{1,3}([0,T]\times \R^m; \R^{m,d})$. Hypothesis~\ref{hyp:F,G} is supposed to hold true. The flow $\phi$ is defined by \eqref{eq:flow-a}. Then a path
$v \in \cq_{\bzeta}^{\ka}$ is a rough viscosity supersolution (resp. subsolution) of \eqref{eq:hjb-i} if and only if the path $\hat{v} = v \circ \phi$ is a viscosity supersolution (resp. subsolution) of 
\beq\label{eq:flow-det}
\hat{v}_t(y) = \hat{v}_0(y) - \int_0^t \sup_{\ga \in \ck} \hat{H}(u, y, \ga, \nabla \hat{v}_u(y)) du,
\eeq
where the random Hamiltonian $\hat{H}$ is defined on $[0,T]\times \mathbb R\times \mathcal K \times \mathbb R^m$ by
\beq\label{eq:H^}
\hat{H} (u, y, \ga, p) = {H}(u, \phi_u(y), \ga, p \cdot \nabla \phi_u^{-1}(\phi_u(y))).
\eeq
\end{proposition}

\begin{proof}\sloppy
We just prove that if $v$ is a rough supersolution of \eqref{eq:hjb-i}, then $\hat{v} = v \circ \phi$ is a supersolution of \eqref{eq:flow-det}. The other implications are shown in a similar way. Let $v$ be a rough viscosity supersolution of \eqref{eq:hjb-i}. From Definition~\ref{def:j-f}, this implies that $v_T(y) \geq G(y)$. In addition, for every element $\psi \in \ct_{\si}$ such that $v - \psi$ admits a local minimum at $(s,y)$, the drift $\psi^t$ satisfies
$$
\psi_s^t(y) \leq -\sup_{\ga \in \ck} H(s,y, \ga, \nabla \psi_s(y)).
$$

Now consider $\hat{v} = v\,\circ\,\phi$ for $\phi$ satisfying the flow in equation~\eqref{eq:flow-a}. The problem of showing that $\hat{v}$ is a viscosity solution of \eqref{eq:flow-det} is now reduced to showing that for any $\hat{\psi} \in \ct_0$ such that $\hat{v} - \hat{\psi}$ reaches a minimum at $(s,y)$, the drift $\hat{\psi}^t$ verifies the relation $\hat{\psi}_s^t(y) \leq -\hat{H}(t,y,\nabla \hat{v}_u(y))$. 

To that effect notice that owing to Proposition~\ref{prop:flow-c}, every $\hat{\psi} \in \ct_0$ can be written as $\hat{\psi} = \psi \circ \phi$ with $\psi \in \ct_{\si}$. Now let us simply explicit the fact that $v$ is a supersolution of \eqref{eq:hjb-i}. It means that for an element  $\psi = \hat{\psi} \circ \phi^{-1} \in \ct_{\si}$, if $v - \psi$ admits a local minimum at $(s,z)$ then $\psi^t$ is such that
\beq\label{eq:f-g}
\psi_s^t(z) \leq - \sup_{\ga \in \ck} H(s, z, \ga, \nabla \psi_s(z)).
\eeq
Since $\phi_t$ is a diffeomorphism for all $t \geq 0$, consider now a $z$ such that $y = \phi^{-1}_s(z)$, so that $z = \phi_s(y)$. Relation \eqref{eq:f-g} becomes
\beq\label{eq:f-h}
\psi_s^t \circ \phi_s(y) \leq -\sup_{\ga \in \ck} H(s, \phi_s(y), \ga, \nabla \psi_s(\phi_s(y))).
\eeq
The left hand side of \eqref{eq:f-h} is then easily identified with $\hat{\psi}_s^t(y)$, since $\hat{\psi} = \psi \circ \phi$. In addition an elementary composition rule yields
$$
\nabla \hat{\psi}_s (y) = \nabla (\psi_s \circ \phi_s)(y) = \nabla \psi_s(\phi_s(y)) \cdot \nabla \phi_s(y).
$$
Hence the Hamiltonian in the right hand side of \eqref{eq:f-h} can be recast as
\begin{eqnarray}
\label{eq:f-i}
H(s, \phi_s(y), \ga, \nabla\psi_s(\phi_s(y))) &=& 
H(s, \phi_s(y), \ga, \nabla \hat{\psi}_s(y)\cdot (\nabla \phi_s(y))^{-1}) \notag\\ 
&=& \hat{H}(s, y, \ga, \nabla \hat{\psi}_s(y)),
\end{eqnarray}
where the last identity stems from our definition \eqref{eq:H^}. Plugging relation~\eqref{eq:f-i} into \eqref{eq:f-h} we end up with
\beq\label{eq:f-j}
\hat{\psi}_s^t(y) \leq \sup_{\ga \in \ck} \hat{H} (s,y,\ga,\nabla \hat{\psi}_s(y)).
\eeq
We have thus obtained that for every $\hat{\psi} \in \ct_0$ such that $\hat{v} - \hat{\psi}$ admits a local minimum at $(s,y)$, the bound \eqref{eq:f-j} holds true. This proves that whenever $v$ is a supersolution of \eqref{eq:hjb-i}, the function $\hat{v} = v \circ \phi$ is a supersolution of \eqref{eq:flow-det}. Our claim is achieved.
\end{proof}

We are now ready to prove our main existence and uniqueness for the Hamilton-Jacobi equation~\eqref{eq:hjb-i}.
\begin{theorem}~\label{thm:ex-uniq}
Consider a coefficient $\si \in \cac_b^{1,3}([0,T] \times \R^m; \R^{m,d})$, as well as a reward function $F$ satisfying Hypothesis~\ref{hyp:F,G}. The path $\xi$ is assumed to fulfill Hypothesis~\ref{hyp:zeta}. Then,  equation~\eqref{eq:hjb-i} admits a unique rough viscosity solution in the sense of Definition~\ref{def:j-f}. 
\end{theorem}

\begin{proof}
We have seen in Theorem~\ref{thm:v-hjb} that the value function $V$ is a rough viscosity solution of~\eqref{eq:hjb-i} according to Definition~\ref{def:j-f}. This proves in particular the existence part of our theorem. We will thus focus on the question of uniqueness. Furthermore, in light of Proposition~\ref{prop:flow-det} the uniqueness of \eqref{eq:hjb-i} is equivalent to the uniqueness of equation~\eqref{eq:flow-det}. In order to deal with~\eqref{eq:flow-det}, recall that the random Hamiltonian $\hat H$ is defined by~\eqref{eq:H^} and introduce the quantity 
\begin{align}\label{f1}
	\tilde H(t,x,p) = \sup_{\gamma \in \mathcal K} \hat H(t,x,\gamma,p).
\end{align}
According to \cite[Ch.~4~Th.~2.5]{yong-zhou} a sufficient condition for the uniqueness of equation~\eqref{eq:hjb-i} are the following inequalities involving the Hamiltonian $\tilde{H}(t,x,p)$:
\beq\label{eq:H-bnds}
 \lln \tilde{H}(t,x,p) - \tilde{H}(t,y,p) \rrn \lesssim (1+|p|)|x-y|,  \quad \lln \tilde{H}(t,x,p) - \tilde{H}(t,x,q) \rrn \lesssim |p-q|,
\eeq
for all $t \in [0,T]$ and $x,y,p,q \in \R^m$. The proof of our theorem is thus reduced to showing~\eqref{eq:H-bnds}. 
To that effect, let us cite the following result from \cite[Proposition~11.13]{friz-victoir} which states that
\beq\label{eq:diff-pd-bnds}
\sup_{t \in [0,T]} \sup_{\eta \in \R^m} \lcl \lln \nabla \phi_t(\eta) \rrn, \lln \nabla \phi_t^{-1}(\eta) \rrn, \lln \nabla^2 \phi_t^{-1}(\eta) \rrn \rcl \leq C_{\bzeta, \si, T}, 
\eeq
for some constant $C_{\bzeta, \si, T}$ depending on these arguments.

In order to show the first inequality in~\eqref{eq:H-bnds}, we invoke~\eqref{eq:hjb-iii} and use some elementary manipulations of the supremum over $\gamma$. We get
\begin{equation}\label{eq:f2}
	\lln \tilde H(t,x,p) - \tilde H(t,y,p) \rrn \leq \mathcal{S}_t^1(x,y) + \mathcal S_{t,p}^2(x,y) \, ,
\end{equation}
where the quantities $\mathcal S_t^1$ and $\mathcal S_{t,p}^2$ are respectively defined by 
\begin{align}
	\label{eq:f3}
	 \mathcal{S}_t^1(x,y) &= \sup_{\gamma \in \mathcal K} \lln F(t,\phi_t(x),\gamma) - F(t,\phi_t(y),\gamma) \rrn\\
	 \label{eq:f4}
	 \mathcal{S}_{t,p}^2(x,y) &= \sup_{\gamma \in \mathcal K} \lln p \cdot \lp \nabla\phi_t^{-1}(\phi_t(x)) - \nabla \phi_t^{-1}(\phi_t(y)) \rp \int_U b(t, \phi_t(x), a) \ga(da) \rrn.
\end{align}
Next the term $\mathcal S_t^1$ in \eqref{eq:f3} is easily bounded thanks to our uniform Lipschitz assumption on $F$ in Hypothesis~\ref{hyp:F,G}. We get
\begin{align}\label{eq:f5}
	\mathcal S_t^1(x,y) \lesssim \lln \phi_t(x) - \phi_t(y) \rrn. 
\end{align}
The right side of \eqref{eq:f5} is now bounded using from \eqref{eq:diff-pd-bnds} the fact that $\nabla \phi $  is bounded to get
\begin{align} 
	\label{eq:f6}	
	\mathcal S_t^1(x,y) \leq C_{\bzeta, \si, T} |x-y|.
\end{align}
The term $\mathcal S^2$ in \eqref{eq:f4} can be handled similarly. Indeed, resorting to the boundedness of $b$ and from \eqref{eq:diff-pd-bnds} the fact that $\nabla^2 \phi^{-1}$ is bounded we obtain 
\begin{align}
	\label{eq:f7}
	\mathcal S_{t,p}^2(x,y) \leq C_{\bzeta, \si, T} |p| |x-y|.
\end{align}
This proves the first part of~\eqref{eq:H-bnds}. 

The second inequality in \eqref{eq:H-bnds} is obtained in a similar way. Namely, we start from expressions~\eqref{f1} and~\eqref{eq:H^} and using an elementary inequality involving suprema we get
\begin{multline*}
 \lln \tilde{H}(t,x,p) - \tilde{H}(t,x,q) \rrn \\
 \leq \sup_{\ga \in \ck} \lln H(t, \phi_t(x),\ga, p \cdot \nabla \phi_t^{-1}(\phi_t(x))) - H(t, \phi_t(x), \ga, q\cdot \nabla \phi_t^{-1}(\phi_t(x))) \rrn.
\end{multline*}
Moreover, recall that $H$ is introduced in~\eqref{eq:hjb-iii}. Therefore, we obtain
\begin{equation*}
\lln \tilde{H}(t,x,p) - \tilde{H}(t,x,q) \rrn
\le \sup_{\ga \in \ck} \lln (p-q) \cdot \nabla \phi_t^{-1}(\phi_t(x))\int_U b(t, \phi_t(x), a) \ga(da) \rrn .
\end{equation*}
Again, with the help of boundedness of $b$ and from \eqref{eq:diff-pd-bnds} the fact that $\nabla \phi^{-1}$ is bounded we are able to write 
\begin{equation}\label{eq:H^-2}
\lln \tilde{H}(t,x,p) - \tilde{H}(t,x,q)\rrn 
\leq C_{\bzeta, \si, T}
 |p-q| \, .
\end{equation}

Eventually, gathering \eqref{eq:f6} and \eqref{eq:f7} into \eqref{eq:f2}, we achieve the second inequality in~\eqref{eq:H-bnds} and finish our proof.
\end{proof}

\subsection{Smooth Approximations}
One of the advantages of a rough path setting for the value function $V$ in \eqref{eq:hjb-d} or the HJB equation~\eqref{eq:hjb-i} is that it easily allows limiting procedures along smooth sequences of driving processes $\bzeta^n$. We state a result, which is reminiscent of \cite[Theorem~5]{diehl}, in this direction.
\begin{theorem}\label{thm:smooth-approx}
Let $\{\bzeta^n; n \geq 1\}$ be a sequence of differentiable paths converging to $\bzeta$ in the rough paths norm given by \eqref{eq:zeta-norm}. For a fixed $n \geq 1$ we consider a family $x^{y, \ga, n}$ of controlled equations driven by $\zeta^n$, with $y \in \R^m$ and $\ga \in \cv^{\ep,L}$:
$$
x_t^{y, \ga, n} = y + \int_0^t \int_U b(r, x_r^{y, \ga, n}, a) \ga_r(da)dr + \int_0^t \si(r, x_r^{y, \ga, n}) d\bzeta_r^n.
$$
The corresponding value is defined similarly to \eqref{eq:hjb-a} by
\beq\label{eq:V^n}
V^n(s, y) = \sup \lcl J_{0T}^n (\ga, y); \ga \in \cv^{\ep,L} \rcl,
\eeq
where $J_{0T}^n$ is the analog of \eqref{eq:hjb-e}:
$$
J_{0T}^n(\ga, y) = \int_0^T F(r, x_r^{y, \ga, n}, \ga_r) dr + G(x_T^{y, \ga, n}).
$$
Then under the hypotheses of Theorem~\ref{thm:ex-uniq}, it follows that $V^n \to V$ locally uniformly on $[0,T]\times \mathbb R^m$ as $n \to \infty$.
\end{theorem}

\begin{proof}
We simply use the definition~\eqref{eq:V^n} of $V^n$, and recall that the value $V$ is given by \eqref{eq:hjb-a}. Then we easily get
\begin{align*}
&|V^n(t,y) - V(t,y)| \leq \sup_{\gamma \in \mathcal P(U)} \left| J_{0T}(\gamma,y,\zeta^n) - J_{0T}(\gamma,y,\zeta) \right|\\
&\leq \sup_{\gamma \in \mathcal P(U)}  \int_0^T \left| F(r, x_r^{y,\gamma,\zeta^{n}},\gamma_r) - F(r, x_r^{y,\gamma,\zeta},\gamma_r) \right| dr + \left| G(x_T^{y,\gamma,\zeta^n}) - G(x_T^{y,\gamma,\zeta}) \right|.
\end{align*}
By Hypothesis~\ref{hyp:F,G}, $F(s,x,\gamma)$ is bounded, continuous in $x$ uniformly over $\gamma \in \mathcal P(U)$ and $G$ is bounded, continuous. Since $b$ and $\sigma$ are sufficiently regular, it follows that $x_r^{y,\gamma,\zeta^n} \to x_r^{y,\gamma,\zeta}$ as $\zeta^n \to \zeta$ in the $\alpha$-H\"older rough path metric. Therefore, it follows from these observations that $V^n$ converges to $V$ locally uniformly on $[0,T]\times \mathbb R^m$ as $n \to \infty$.
\end{proof}

\begin{remark}
The approach to rough viscosity solutions in \cite{diehl} was essentially based on limits along smooth sequences of paths. Our point of view in the current paper is different: we have based our analysis on an intrinsic notion of rough viscosity solution developed in Sections~\ref{sec:visc}, \ref{sec:flows} and~\ref{sec:visc-uniq}. Now for a fixed $n$ it is classical (see \cite[Theorem~12.5.3]{BerkowitzMedhin}) that the value $V^n$ in~\eqref{eq:V^n} is the unique viscosity solution of 
\beq\label{eq:v^n}
\partial_t v(t,y) + \sup_{\gamma \in \mathcal P(U)} H(t,y,\gamma,\nabla v(t,y)) + \nabla v(t,y) \cdot \sigma(t,y) d\zeta_t^n.
\eeq
Our notion of rough viscsoity solution certainly allows to take limits directly on \eqref{eq:v^n} and show that $v^n$ converges to the solution of \eqref{eq:hjb-i}. However, the proof of this claim would be more intricate than the simple arguments exposed in the proof of Theorem~\ref{thm:smooth-approx}.

\end{remark}

\section*{Acknowledgements}
\noindent
All three authors are partially supported for this work by NSF grant  DMS-2153915. H. Honnappa is also partly supported through NSF grant CMMI-22014426.

\bigskip

\bibliographystyle{plain}
\bibliography{problem-1-34.bib}

\begin{thebibliography}{10}

\bibitem{allan-cohen}
Andrew~L. Allan and Samuel~N. Cohen.
\newblock {Pathwise stochastic control with applications to robust filtering}.
\newblock {\em The Annals of Applied Probability}, 30(5):2274 -- 2310, 2020.

\bibitem{AramanGlynn}
Victor~F Araman and Peter~W Glynn.
\newblock Fractional brownian motion with h< 1/2 as a limit of scheduled
  traffic.
\newblock {\em Journal of Applied Probability}, 49(3):710--718, 2012.

\bibitem{bailleul-19}
Isma{\"e}l Bailleul and Sebastian Riedel.
\newblock Rough flows.
\newblock {\em Journal of the Mathematical Society of Japan}, 71(3):915--978,
  2019.

\bibitem{BerkowitzMedhin}
Leonard~David Berkovitz and Negash~G Medhin.
\newblock {\em Nonlinear optimal control theory}.
\newblock CRC press, 2012.

\bibitem{bhauryal}
Neeraj Bhauryal, Ana~Bela Cruzeiro, and Carlos Oliveira.
\newblock Pathwise stochastic control and a class of stochastic partial
  differential equations.
\newblock {\em arXiv preprint arXiv:2301.09214}, 2023.

\bibitem{buckdahn}
Rainer Buckdahn, Christian Keller, Jin Ma, and Jianfeng Zhang.
\newblock Fully nonlinear stochastic and rough pdes: Classical and viscosity
  solutions.
\newblock {\em Probability, Uncertainty and Quantitative Risk}, 5(1):1--59,
  2020.

\bibitem{Dai}
Hongshuai Dai.
\newblock Tandem fluid queue with long-range dependent inputs: sticky behaviour
  and heavy traffic approximation.
\newblock {\em Queueing Systems}, 101(1-2):165--196, 2022.

\bibitem{diehl}
Joscha Diehl, Peter~K Friz, and Paul Gassiat.
\newblock Stochastic control with rough paths.
\newblock {\em Applied Mathematics \& Optimization}, 75:285--315, 2017.

\bibitem{kunita}
RM~Dudley, H~Kunita, F~Ledrappier, and H~Kunita.
\newblock Stochastic differential equations and stochastic flows of
  diffeomorphisms.
\newblock In {\em Ecole d'{\'e}t{\'e} de probabilit{\'e}s de Saint-Flour
  XII-1982}, pages 143--303. Springer, 1984.

\bibitem{fleming-nisio-b}
Wendell~H Fleming and Makiko Nisio.
\newblock On the existence of optimal stochastic controls.
\newblock {\em Journal of Mathematics and Mechanics}, 15(5):777--794, 1966.

\bibitem{fleming-nisio-a}
Wendell~H Fleming and Makiko Nisio.
\newblock On stochastic relaxed control for partially observed diffusions.
\newblock {\em Nagoya Mathematical Journal}, 93:71--108, 1984.

\bibitem{friz-hoquet-le}
Peter~K Friz, Antoine Hocquet, and Khoa L{\^e}.
\newblock Rough stochastic differential equations.
\newblock {\em arXiv preprint arXiv:2106.10340}, 2021.

\bibitem{friz-victoir}
Peter~K Friz and Nicolas~B Victoir.
\newblock {\em Multidimensional stochastic processes as rough paths: theory and
  applications}, volume 120.
\newblock Cambridge University Press, 2010.

\bibitem{Gatheral}
Jim Gatheral, Thibault Jaisson, and Mathieu Rosenbaum.
\newblock Volatility is rough.
\newblock In {\em Commodities}, pages 659--690. Chapman and Hall/CRC, 2022.

\bibitem{GhoshWeera}
Arka~P Ghosh, Alexander Roitershtein, and Ananda Weerasinghe.
\newblock Optimal control of a stochastic processing system driven by a
  fractional brownian motion input.
\newblock {\em Advances in Applied Probability}, 42(1):183--209, 2010.

\bibitem{gubinelli}
Massimiliano Gubinelli.
\newblock Controlling rough paths.
\newblock {\em Journal of Functional Analysis}, 216(1):86--140, 2004.

\bibitem{E1}
Jiequn Han, Qianxiao Li, et~al.
\newblock A mean-field optimal control formulation of deep learning.
\newblock {\em Research in the Mathematical Sciences}, 6(1):1--41, 2019.

\bibitem{Hoglund}
Melker Hoglund, Emilio Ferrucci, Camilo Hernandez, Aitor~Muguruza Gonzalez,
  Cristopher Salvi, Leandro Sanchez-Betancourt, and Yufei Zhang.
\newblock A neural rde approach for continuous-time non-markovian stochastic
  control problems.
\newblock {\em arXiv preprint arXiv:2306.14258}, 2023.

\bibitem{Kidger2}
Patrick Kidger.
\newblock On neural differential equations.
\newblock {\em arXiv preprint arXiv:2202.02435}, 2022.

\bibitem{Kidger1}
Patrick Kidger, James Morrill, James Foster, and Terry Lyons.
\newblock Neural controlled differential equations for irregular time series.
\newblock {\em Advances in Neural Information Processing Systems},
  33:6696--6707, 2020.

\bibitem{KimYang}
Jeongho Kim and Insoon Yang.
\newblock Maximum entropy optimal control of continuous-time dynamical systems.
\newblock {\em IEEE Transactions on Automatic Control}, 68(4):2018--2033, 2022.

\bibitem{KolShai}
Vladimir~Borisovich Kolmanovski{\u\i} and Leonid~Efimovich Sha{\u\i}khet.
\newblock {\em Control of systems with aftereffect}, volume 157.
\newblock American Mathematical Soc., 1996.

\bibitem{Kurtz}
Thomas~G Kurtz.
\newblock Limit theorems for workload input models, 1996.

\bibitem{Leao}
Dorival Le{\~a}o, Alberto Ohashi, and Francys~Andrews de~Souza.
\newblock Solving non-markovian stochastic control problems driven by wiener
  functionals.
\newblock {\em arXiv preprint arXiv:2003.06981}, 2020.

\bibitem{LeeWeera}
Chihoon Lee and Ananda Weerasinghe.
\newblock Stationarity and control of a tandem fluid network with fractional
  brownian motion input.
\newblock {\em Advances in Applied Probability}, 43(3):847--874, 2011.

\bibitem{lewis-vintner}
Richard~M Lewis and Richard~B Vinter.
\newblock Relaxation of optimal control problems to equivalent convex programs.
\newblock {\em Journal of Mathematical Analysis and Applications},
  74(2):475--493, 1980.

\bibitem{lions-souganidis}
Pierre-Louis Lions and Panagiotis~E Souganidis.
\newblock Fully nonlinear stochastic partial differential equations: non-smooth
  equations and applications.
\newblock {\em Comptes Rendus de l'Acad{\'e}mie des Sciences-Series
  I-Mathematics}, 327(8):735--741, 1998.

\bibitem{mazliak-nourdin}
Laurent Mazliak and Ivan Nourdin.
\newblock Optimal control for rough differential equations.
\newblock {\em Stoch. Dyn.}, 8(1):23--33, 2008.

\bibitem{mcshane}
EJ~McShane.
\newblock Relaxed controls and variational problems.
\newblock {\em SIAM Journal on Control}, 5(3):438--485, 1967.

\bibitem{Morrill}
James Morrill, Cristopher Salvi, Patrick Kidger, and James Foster.
\newblock Neural rough differential equations for long time series.
\newblock In {\em International Conference on Machine Learning}, pages
  7829--7838. PMLR, 2021.

\bibitem{munkres}
J~R Munkres.
\newblock {\em Elements of Algebraic Topology}.
\newblock Addison-Wesley, 1984.

\bibitem{Neu}
Gergely Neu, Anders Jonsson, and Vicen{\c{c}} G{\'o}mez.
\newblock A unified view of entropy-regularized markov decision processes.
\newblock {\em arXiv preprint arXiv:1705.07798}, 2017.

\bibitem{elkaroui}
Karoui Nicole~el, Nguyen Du'h{\=U}{\=U}, and Jeanblanc-Picqu{\'e} Monique.
\newblock Compactification methods in the control of degenerate diffusions:
  existence of an optimal control.
\newblock {\em Stochastics: an international journal of probability and
  stochastic processes}, 20(3):169--219, 1987.

\bibitem{polyanski-wu}
Yury Polyanskiy and Yihong Wu.
\newblock Wasserstein continuity of entropy and outer bounds for interference
  channels.
\newblock {\em IEEE Transactions on Information Theory}, 62(7):3992--4002,
  2016.

\bibitem{roubicek}
Tom{\'a}{\v{s}} Roub{\'\i}{\v{c}}ek.
\newblock {\em Relaxation in optimization theory and variational calculus},
  volume~4.
\newblock Walter de Gruyter GmbH \& Co KG, 2020.

\bibitem{souganidis-course-cetraro}
Panagiotis~E Souganidis.
\newblock Pathwise solutions for fully nonlinear first-and second-order partial
  differential equations with multiplicative rough time dependence.
\newblock {\em Singular random dynamics}, 2253:75--220, 2019.

\bibitem{Sutskever}
Ilya Sutskever.
\newblock {\em Training recurrent neural networks}.
\newblock University of Toronto Toronto, ON, Canada, 2013.

\bibitem{tang-zhang-zhou}
Wenpin Tang, Yuming~Paul Zhang, and Xun~Yu Zhou.
\newblock Exploratory hjb equations and their convergence.
\newblock {\em SIAM Journal on Control and Optimization}, 60(6):3191--3216,
  2022.

\bibitem{wang-zari-zhou}
Haoran Wang, Thaleia Zariphopoulou, and Xun~Yu Zhou.
\newblock Reinforcement learning in continuous time and space: A stochastic
  control approach.
\newblock {\em The Journal of Machine Learning Research}, 21(1):8145--8178,
  2020.

\bibitem{E2}
E~Weinan.
\newblock A proposal on machine learning via dynamical systems.
\newblock {\em Communications in Mathematics and Statistics}, 1(5):1--11, 2017.

\bibitem{yong-zhou}
Jiongmin Yong and Xun~Yu Zhou.
\newblock {\em Stochastic controls: Hamiltonian systems and HJB equations},
  volume~43.
\newblock Springer Science \& Business Media, 2012.

\end{thebibliography}

\end{document}